\documentclass[10pt,a4paper]{article}
\usepackage[utf8]{inputenc}
\usepackage{amsmath}
\usepackage{pifont}
\usepackage{bm}
\usepackage[american]{babel}
\usepackage[all]{xy}
\usepackage{enumitem}
\usepackage{anysize}
\usepackage[dvipsnames]{xcolor}
\usepackage{graphicx}
\usepackage{stmaryrd}
\usepackage{amsfonts,amstext,amssymb,amsmath,amsthm,pspicture,bigints}

\usepackage{fullpage}
\usepackage{moreverb}
\usepackage{pifont}
\usepackage{pst-all}
\usepackage[left=2cm,right=2cm,top=2cm,bottom=2cm]{geometry}
\usepackage{esint}
\usepackage{fourier-orns}
\usepackage{mathrsfs}
\usepackage{array}

\usepackage{hyperref}

\usepackage{authblk}
\newcolumntype{C}[1]{>{\centering\arraybackslash}b{#1}}
\newcolumntype{R}[1]{>{\raggedleft\arraybackslash}b{#1}}
\newcolumntype{L}[1]{>{\raggedright\arraybackslash}b{#1}}
\newcolumntype{M}[1]{>{\centering}m{#1}}
\newtheorem{theo}{Theorem}[section]
\newtheorem{defin}{Definition}[section]
\newtheorem{lem}{Lemma}[section]
\newtheorem{prop}{Proposition}[section]

\newtheorem{remark}{Remark}[section]

\usepackage{bbm}
\numberwithin{equation}{section}
\usepackage{tikz}
\usepackage{mathtools}
\DeclareMathAlphabet{\mathpzc}{OT1}{pzc}{m}{it}

\date{}
\title{Ions-electrons-states for the two-component Vlasov-Poisson equation}
\author{Emeric Roulley\thanks{Università degli Studi di Milano (UniMi), Dipartimento di Matematica Federigo Enriques, Via Cesare Saldini, 50, 20133 Milano, Italy.\\
			E-mail address : emeric.roulley@unimi.it
			}}
\begin{document}
	\maketitle
	\begin{abstract}
		We establish both local and global bifurcation results for traveling periodic solutions of the one-dimensional two-species Vlasov–Poisson equation. These solutions consist of strip-like regions of ions and electrons in phase space that propagate coherently and emerge from spatially homogeneous, velocity-dependent equilibrium layers. Depending on the geometry of the underlying equilibrium and on the selected Fourier mode, the bifurcation diagram exhibits either two or four solution branches. In all cases, the bifurcation is of pitchfork type; in symmetric configurations, the local structure near the equilibrium has a hyperbolic geometry. We further show that these locally constructed branches extend globally. This work extends the previous study \cite{R23} of the purely electronic case, where the ions were modeled as an immobile neutralizing background. Allowing both species to evolve dynamically leads to a more intricate, higher-dimensional analysis. Finally, by means of an affine change of variables, we reveal a connection with the one-dimensional two-component Euler–Poisson system, which in turn enables the construction of traveling periodic waves of both small and large amplitude for that model as well.
	\end{abstract}
	\tableofcontents
	\section{Introduction}
	In this article, we investigate some dynamical properties of plasmas made of ions and electrons. We shall first present the equations of interest in this study, called \textit{two-component Vlasov-Poisson system}. Then, we may discuss a particular class of solutions called \textit{ions-electrons layers} for which we find a new set of equations describing the perturbative regime near stationary trivial solutions. Finally, we expose our main results.
	\subsection{The two-component Vlasov-Poisson system}
	The Vlasov--Poisson system constitutes one of the fundamental kinetic models in plasma physics and galactic dynamics, describing the evolution of collisionless charged particle distributions under a self-consistent electrostatic field.
Let us consider a collisionless neutral plasma composed with ions and electrons and evolving in a single dimension of space. This latter fact implies that the particle motion is only influenced by induced electrostatic forces and therefore we disregard electromagnetic interactions. We further assume that the space direction $x$ is $2\pi$-periodic, i.e. $x\in\mathbb{T}\triangleq\mathbb{R}/2\pi\mathbb{Z}.$ The velocity variable is denoted $v.$  We consider $f_+(t,x,v)\geqslant0$ and $f_-(t,x,v)\geqslant0$ respectively the distribution of ions and electrons traveling with speed $v$ at position $x$ and time $t.$ The dynamics of these two quantities is described by the two-component Vlasov-Poisson system, see \cite[Chap. 13]{BM02},
\begin{equation}\label{Two component VP}
	\begin{cases}
		\partial_{t}f_+(t,x,v)+v\partial_{x}f_+(t,x,v)+E(t,x)\partial_{v}f_+(t,x,v)=0,\\
		\partial_{t}f_-(t,x,v)+v\partial_{x}f_-(t,x,v)-E(t,x)\partial_{v}f_-(t,x,v)=0,
	\end{cases}\qquad(t,x,v)\in\mathbb{R}_+\times\mathbb{T}\times\mathbb{R},
\end{equation}
where $E(t,x)$ is the electric field given by
\begin{equation}\label{electric field and potential}
	E(t,x)=\partial_{x}\boldsymbol{\varphi}(t,x),\qquad\partial_{xx}\boldsymbol{\varphi}(t,x)=\int_{\mathbb{R}}\big(f_+(t,x,v)-f_-(t,x,v)\big)dv.
\end{equation}
By using \eqref{electric field and potential} and Taylor formula, the periodic condition
$$E(t,0)=E(t,2\pi)$$
can be reformulated into the following neutrality condition
\begin{equation}\label{neutrality condition}
	\int_0^{2\pi}\int_{\mathbb{R}}f_+(t,x,v)dxdv=\int_0^{2\pi}\int_{\mathbb{R}}f_-(t,x,v)dxdv.
\end{equation}
The quantity $\boldsymbol{\varphi}$ in \eqref{electric field and potential} is the electric potential. Notice that only the quantity $\partial_{x}\boldsymbol{\varphi}$ is of interest in the problem. Therefore, the electric potential $\boldsymbol{\varphi}$ is well-defined up to a time dependent additive constant that we select so that, for any time, $\boldsymbol{\varphi}$ has zero space average. Consequently, introducing the inverse Laplace operator $\partial_{xx}^{-1}$ defined in Fourier as
$$\forall j\in\mathbb{Z}^*,\quad\partial_{xx}^{-1}\cos(jx)=-\frac{\cos(jx)}{j^2}\qquad\textnormal{and}\qquad\partial_{xx}^{-1}\sin(jx)=-\frac{\sin(jx)}{j^2},$$
we obtain from \eqref{electric field and potential} and \eqref{neutrality condition} that
\begin{equation}\label{inv Lap phi}
	\boldsymbol{\varphi}(t,x)=\partial_{xx}^{-1}\left(\int_{\mathbb{R}}\big(f_+(t,x,v)-f_-(t,x,v)\big)dv\right).
\end{equation}
Functions depending on the velocity variable only $(f_{\pm}(t,x,v)=f_{\pm}(v)$) consitute trivial solutions with zero electric field that are called \textit{homogeneous states}. From the mathematical viewpoint, the Vlasov--Poisson system has been extensively studied with regard to well-posedness and long-time behavior. Global existence of classical solutions was established in various settings by Bardos--Degond \cite{BD85}, Pfaffelmoser \cite{P92}, and Schaeffer \cite{S91}. In 1946, Landau \cite{L46} observed near Maxwellians (Gaussians homogeneous states) a damping effect, namely time decay of the electric field. It has been mathematically rigorously justified in the work of Mouhot--Villani \cite{MV11} and subsequent contributions that small perturbations of homogeneous states may damp out in analytic or Gevrey spaces. Lowering the regularity it is possible to find nontrivial steady or traveling structures. Indeed, there is a huge physiscal literature of traveling waves for the Vlasov-Poisson system \cite{AM67,BGK57,BG49,BD95,G70,GIBFFS88,HRK94,K55,MB00,PA14,SLAS79}. Among them, a cornerstone in the study of nonlinear electrostatic waves is the seminal work of Bernstein, Greene, and Kruskal \cite{BGK57}, who introduced what are now known as BGK waves. These are nontrivial stationary or traveling solutions constructed by prescribing the distribution function as a function of the particle energy. Their approach yields spatially periodic or solitary phase-space structures and demonstrates that the Vlasov--Poisson system admits a rich family of nonlinear equilibria beyond homogeneous states. The BGK waves constitute an obstruction to damping in weaker topologies. This interplay has motivated a more systematic investigation of nonlinear periodic and traveling solutions. The rigorous construction, stability analysis and caracterizations of BGK waves has been addressed in several works and the reader is refered to \cite{BGHP26,GL17,GS95,GS98,L01,L05,LZ11,STZ25}. These results reveal that periodic kinetic waves can exhibit subtle dynamical behavior and are often spectrally or nonlinearly unstable, emphasizing the need for a careful structural analysis of traveling solutions. In the present work, we construct traveling periodic waves that live near singular homogeneous states.

	\subsection{Ions and electrons layers}
	Our next goal is to present to notion of patches of ions/electrons which are weak solutions corresponding to localized regions of the phase space $(x,v)$ evolving in time. Among them, we may focus on the particular subclass of ions/electrons layers that are patch-type solutions with strip-shaped domains.\\
	
	Let us recast the equations \eqref{Two component VP} as a system of two coupled active scalar equations. For that, we assimilate the phase space $\mathbb{T}\times\mathbb{R}$ to the cylinder manifold embedded in $\mathbb{R}^3$ through the local chart
	\begin{equation}\label{local chart}
		\psi_1:\begin{array}[t]{rcl}
			(0,1)\times\mathbb{R} & \rightarrow & \mathbb{R}^3\\
			(x,v) & \mapsto & \big(\cos(x),\sin(x),v\big).
		\end{array}
	\end{equation}
	Using the classical identification vector/directional derivative we get that at any point $(x,v)\in\mathbb{T}\times\mathbb{R},$ the tangent plane $T_{(x,v)}(\mathbb{T}\times\mathbb{R})\equiv\mathbb{R}^2$ admits the orthonormal basis
	$$\mathtt{e}_{x}\triangleq\partial_{x},\qquad\mathtt{e}_{v}\triangleq\partial_{v}.$$
	For any function $\mathtt{g}:\mathbb{T}\times\mathbb{R}\rightarrow\mathbb{R},$ the gradient is given by
	$$\nabla_{x,v}\mathtt{g}(x,v)=\partial_{x}\mathtt{g}(x,v)\mathtt{e}_{x}+\partial_{v}\mathtt{g}(x,v)\mathtt{e}_{v}.$$
	The orthogonal gradient is obtained by a rotation of angle $\tfrac{\pi}{2}$
	\begin{equation}\label{def:gradperp}
		\nabla_{x,v}^{\perp}\triangleq\mathtt{J}_{x,v}\nabla_{x,v},\qquad\underset{(\mathtt{e}_{x},\mathtt{e}_{v})}{\textnormal{Mat}}(\mathtt{J}_{x,v})=\begin{pmatrix}
			0 & -1\\
			1 & 0
		\end{pmatrix}.
	\end{equation}
	Consider the velocity field
	\begin{equation}\label{def:velocity fields}
		\mathbf{v}_{\pm}:\begin{array}[t]{rcl}
			\mathbb{T}\times\mathbb{R} & \rightarrow & \displaystyle T(\mathbb{T}\times\mathbb{R})\triangleq\bigcup_{(x,v)\in\mathbb{T}\times\mathbb{R}}T_{(x,v)}(\mathbb{T}\times\mathbb{R})\\
			(x,v) & \mapsto & v\mathtt{e}_{x}\pm E(t,x)\mathtt{e}_{v},
		\end{array}
	\end{equation}
	which is divergence-free
	\begin{equation}\label{div=0}
		\textnormal{div}_{x,v}\mathbf{v}_{\pm}(t,x,v)=\partial_{x}(v)\pm\partial_{v}\big(E(t,x)\big)=0.
	\end{equation}
	More precisely, by virtue of \eqref{electric field and potential}, we can write
	\begin{equation}\label{vel pot}
		\mathbf{v}_{\pm}(t,x,v)=-\nabla_{x,v}^{\perp}\boldsymbol{\Psi}_{\pm}(t,x,v),\qquad\boldsymbol{\Psi}_{\pm}(t,x,v)\triangleq\frac{v^2}{2}\mp\boldsymbol{\varphi}(t,x).
	\end{equation}
	Then, the system \eqref{Two component VP} can be written in the following coupled system of active scalar equations
	\begin{equation}\label{coupled active scalar}
		\begin{cases}
			\partial_tf_+(t,x,v)+\Big\langle\mathbf{v}_+(t,x,v),\nabla_{x,v}f_+(t,x,v)\Big\rangle_{T_{(x,v)}(\mathbb{T}\times\mathbb{R})}=0,\\
			\partial_tf_-(t,x,v)+\Big\langle\mathbf{v}_-(t,x,v),\nabla_{x,v}f_-(t,x,v)\Big\rangle_{T_{(x,v)}(\mathbb{T}\times\mathbb{R})}=0,
		\end{cases}
	\end{equation}
	where the scalar product on the tangent space $T_{(x,v)(\mathbb{T}\times\mathbb{R})}$ is defined by
	\begin{equation}\label{scalarprod:tangentspace}
		\Big\langle\alpha\mathtt{e}_{x}+\beta\mathtt{e}_{v}\,,\gamma\mathtt{e}_{x}+\delta\mathtt{e}_{v}\Big\rangle_{T_{(x,v)}(\mathbb{T}\times\mathbb{R})}\triangleq \alpha\gamma+\beta\delta.
	\end{equation}
	Following the theory of Dziurzynski \cite{D87}, any initial datum $(f_+^{0},f_-^{0})\in\big(L^{\infty}(\mathbb{T}\times\mathbb{R})\big)^2$ generates a unique global in time weak solution $(f_+,f_-)\in L^{\infty}\left(\mathbb{R}_+,\big(L^{\infty}(\mathbb{T}\times\mathbb{R})\big)^2\right)$ which is Lagrangian, i.e.
	$$f_{\pm}(t,x,v)=f_{\pm}^0\big(X_t^{\pm}(x,v)\big),$$
	where $X_{t}^{\pm}$ is the flow map associated with the velocity field $\mathbf{v}_{\pm}$ in \eqref{def:velocity fields}, that is
	$$\partial_tX_t^{\pm}(x,v)=\mathbf{v}_{\pm}\big(t,X_t^{\pm}(x,v)\big),\qquad X_0^{\pm}(x,v)=(x,v).$$
	In particular, we can describe a particular class of weak solutions given by patches of ions and electrons. Considering two bounded initial domains $\Omega_0^{\pm}\subset\mathbb{T}\times\mathbb{R}$ and an initial datum $(f_+^0,f_-^0)$ in the form  
	$$f_{\pm}^0=\mathbf{1}_{\Omega_0^{\pm}},$$
	then the corresponding weak solution is given by
	$$f_{\pm}(t,\cdot)=\mathbf{1}_{\Omega_t^{\pm}},\qquad\Omega_t^{\pm}\triangleq X_t^{\pm}\big(\Omega_0^{\pm}\big).$$
	Due to the divergence-free property of $\mathbf{v}_{\pm}$  in \eqref{div=0}, we have the following measure preserving property
	\begin{equation}\label{preserv area}
		\forall t\geqslant0,\quad|\Omega_t^{\pm}|=|\Omega_0^{\pm}|.
	\end{equation}
	Moreover, the neutrality condition \eqref{neutrality condition} becomes
	\begin{equation}\label{neutreOMG}
		|\Omega_0^+|=|\Omega_0^-|.
	\end{equation}
	The properties \eqref{preserv area} and \eqref{neutreOMG} mean that, at any time, the ion patch and the electron patch must have the same area.
	In what follows, we shall work with a subclass of patches of ions/electrons called \textit{ions-electrons layers} for being strip shaped, i.e. 
	$$f_{\pm}(t,x,v)=\mathbf{1}_{S_t^{\pm}}(x,v),\qquad S_t^{\pm}\triangleq X_{t}^{\pm}\left(S_0^{\pm}\right),$$
	where the strip $S_t^{\pm}$ is delimited by two $2\pi$-periodic profiles $x\mapsto v_{\pm}^{[k]}(t,x)$ ($k\in\{1,2\}$) such that
	$$S_t^{\pm}=\left\{(x,v)\in\mathbb{T}\times\mathbb{R}\quad\textnormal{s.t.}\quad v_{\pm}^{[1]}(t,x)< v< v_{\pm}^{[2]}(t,x)\right\}.$$
	The boundary of the strips are described by
	$$\partial S_t^{\pm}=\Gamma_{\pm}^{[1]}(t)+\Gamma_{\pm}^{[2]}(t)$$
	where
	\begin{equation}\label{def:parametrizations}
		\forall k\in\{1,2\},\quad\Gamma_{\pm}^{[k]}(t)=\Big\{z_{\pm}^{[k]}(t,x)\triangleq \psi_1\left(x,v_{\pm}^{[k]}(t,x)\right),\,\,x\in\mathbb{T}\Big\}.
	\end{equation}
	Recall that $\psi_1$ is the local chart introduced in \eqref{local chart}. Thanks to the transport structure \eqref{coupled active scalar}, following the computation carried out in \cite[Sec. 1.2]{R23} the dynamics of the patch interfaces is given by: for any $k\in\{1,2\},$
	\begin{equation}\label{vp VP}
		\Big\langle\partial_tz_{\pm}^{[k]}(t,x)\,,\mathtt{J}_{x,v}\partial_{x}z_{\pm}^{[k]}(t,x)\Big\rangle_{T_{z_{\pm}^{[k]}(t,x)}(\mathbb{T}\times\mathbb{R})}=-\partial_{x}\Big(\boldsymbol{\Psi}_{\pm}\big(t,z_{\pm}^{[k]}(t,x)\big)\Big).
	\end{equation}
	We fix four real numbers
	$$a_+^{[1]}<a_+^{[2]}\qquad\textnormal{and}\qquad a_-^{[1]}<a_-^{[2]}$$
	and look for ions/electrons layer solutions so that
	$$\forall k\in\{1,2\},\quad v_{\pm}^{[k]}(t,x)=a_{\pm}^{[k]}+r_{\pm}^{[k]}(t,x),$$ 
	that is
	\begin{equation}\label{ansatz:fpm}
		f_{+}(t,x,v)=\mathbf{1}_{a_+^{[1]}+r_+^{[1]}(t,x)\leqslant v\leqslant a_+^{[2]}+r_+^{[2]}(t,x)},\qquad f_{-}(t,x,v)=\mathbf{1}_{a_-^{[1]}+r_-^{[1]}(t,x)\leqslant v\leqslant a_-^{[2]}+r_-^{[2]}(t,x)}.
	\end{equation}
	Denoting 
	$$\int_{\mathbb{T}}\mathtt{f}(x)dx\triangleq\frac{1}{2\pi}\int_{0}^{2\pi}\mathtt{f}(x)dx,$$
	the neutrality condition \eqref{neutrality condition} reads
	\begin{equation}\label{neutrerworld}
		\Delta_+a+\int_{\mathbb{T}}\big(r_+^{[2]}(t,x)-r_+^{[1]}(t,x)\big)dx=\Delta_-a+\int_{\mathbb{T}}\big(r_-^{[2]}(t,x)-r_-^{[1]}(t,x)\big)dx,
	\end{equation}
	where we have used the notation
	\begin{equation*}
		\Delta_{\pm}a\triangleq a_{\pm}^{[2]}-a_{\pm}^{[1]}.
	\end{equation*}
	Combining \eqref{def:parametrizations}, \eqref{def:gradperp} and \eqref{scalarprod:tangentspace}, we find for any $k\in\{1,2\},$
	\begin{equation}\label{obt:dtr}
		\begin{aligned}
			\Big\langle\partial_tz_{\pm}^{[k]}(t,x)\,,\mathtt{J}_{x,v}\partial_{x}z_{\pm}^{[k]}(t,x)\Big\rangle_{T_{z_{\pm}^{[k]}(t,x)}(\mathbb{T}\times\mathbb{R})}&=\Big\langle\partial_tr_{\pm}^{[k]}(t,x)\mathtt{e}_{v}\,,\mathtt{J}_{x,v}\big(\mathtt{e}_{x}+\partial_{x}r_{\pm}^{[k]}(t,x)\mathtt{e}_{v}\big)\Big\rangle_{T_{z_{\pm}(t,x)}(\mathbb{T}\times\mathbb{R})}\\
			&=\Big\langle\partial_tr_{\pm}^{[k]}(t,x)\mathtt{e}_{v}\,,\mathtt{e}_{v}-\partial_{x}r_{\pm}^{[k]}(t,x)\mathtt{e}_{x}\Big\rangle_{T_{z_{\pm}(t,x)}(\mathbb{T}\times\mathbb{R})}\\
			&=\partial_{t}r_{\pm}^{[k]}(t,x).
		\end{aligned}
	\end{equation}
	Putting together \eqref{vp VP} and \eqref{obt:dtr}, we find
	\begin{equation}\label{IEP:eq}
		\forall k\in\{1,2\},\quad \partial_{t}r_{\pm}^{[k]}(t,x)=-\partial_{x}\Big(\boldsymbol{\Psi}_{\pm}\big(t,z_{\pm}^{[k]}(t,x)\big)\Big).
	\end{equation}
	The space derivative in the right hand-side of each equation of \eqref{IEP:eq} implies that
	$$\forall k\in\{1,2\},\quad\partial_{t}\int_{\mathbb{T}}r_{\pm}^{[k]}(t,x)dx=0.$$
	In the sequel, we make the choice
	\begin{equation}\label{zero space average condition}
		\forall k\in\{1,2\},\quad\int_{\mathbb{T}}r_{\pm}^{[k]}(t,x)dx=0.
	\end{equation}
	Consequently, the neutrality condition \eqref{neutrerworld} becomes
	\begin{equation}\label{equalDeltas}
		\Delta_+a=\Delta_-a,
	\end{equation}
	Inserting \eqref{ansatz:fpm} and \eqref{equalDeltas} into \eqref{inv Lap phi}, we infer
	\begin{equation}\label{electric potential in terms of r}
		\boldsymbol{\varphi}(t,x)=\partial_{xx}^{-1}\left(r_+^{[2]}(t,x)-r_+^{[1]}(t,x)\right)-\partial_{xx}^{-1}\left(r_-^{[2]}(t,x)-r_-^{[1]}(t,x)\right).
	\end{equation}
	
	\begin{figure}[!h]
		\begin{center}
			\begin{tikzpicture}[scale=0.8]
				\draw[->](-1,0)--(8,0)
				node[below right] {$x$};
				\draw[->](0,-1)--(0,5)
				node[left] {$v$};
				\draw[black,dashed] (0,1)--(6.88,1);
				\draw[black,dashed] (0,4)--(6.88,4);
				\draw[black] (6.88,-1)--(6.88,5);
				\node at (-0.4,1) {$a_+^{[1]}$};
				\node at (-0.5,4) {$a_+^{[2]}$};
				\node at (-0.2,-0.2) {$0$};
				\node at (7.1,-0.2) {$2\pi$};
				\draw[domain=0:6.88,thick, black,samples=500] plot [variable=\t] (\t,{4+0.5*sin(50*pi*\t)});
				\draw[domain=0:6.88,thick, black,samples=500] plot [variable=\t] (\t,{1+0.5*cos(50*pi*\t)});
				\draw[red] (2.85,0.1)--(2.85,-0.1);
				\node at (2.85,-0.3) {$\textcolor{red}{x_2}$};
				\node at (2.85,3.5) {$\textcolor{red}{r_+^{[2]}(t,x_2)}$};
				\draw[->,red](2.85,4)--(2.85,4.5);
				\draw[red] (5.75,0.1)--(5.75,-0.1);
				\node at (5.75,-0.3) {$\textcolor{red}{x_1}$};
				\node at (5.75,1.5) {$\textcolor{red}{r_+^{[1]}(t,x_1)}$};
				\draw[->,red](5.75,1)--(5.75,0.5);
			\end{tikzpicture}\begin{tikzpicture}[scale=0.8]
			\draw[->](-1,0)--(8,0)
			node[below right] {$x$};
			\draw[->](0,-2)--(0,3)
			node[left] {$v$};
			\draw[black,dashed] (0,-1)--(6.88,-1);
			\draw[black,dashed] (0,2)--(6.88,2);
			\draw[black] (6.88,-2)--(6.88,3);
			\node at (-0.4,-1) {$a_-^{[1]}$};
			\node at (-0.5,2) {$a_-^{[2]}$};
			\node at (-0.2,-0.2) {$0$};
			\node at (7.2,-0.3) {$2\pi$};
			\draw[domain=0:6.88,thick, black,samples=500] plot [variable=\t] (\t,{2+0.5*sin(50*pi*\t)});
			\draw[domain=0:6.88,thick, black,samples=500] plot [variable=\t] (\t,{-1+0.5*cos(50*pi*\t)});
			\draw[blue] (2.85,0.1)--(2.85,-0.1);
			\node at (2.85,-0.3) {$\textcolor{blue}{x_2}$};
			\node at (2.85,1.5) {$\textcolor{blue}{r_-^{[2]}(t,x_2)}$};
			\draw[->,blue](2.85,2)--(2.85,2.5);
			\draw[blue] (5.75,0.1)--(5.75,-0.1);
			\node at (5.75,-0.3) {$\textcolor{blue}{x_1}$};
			\node at (5.75,-2) {$\textcolor{blue}{r_-^{[1]}(t,x_1)}$};
			\draw[->,blue](5.75,-1)--(5.75,-1.5);
			\end{tikzpicture}
		\end{center}
		\caption{Ion/electron layers}
	\end{figure}
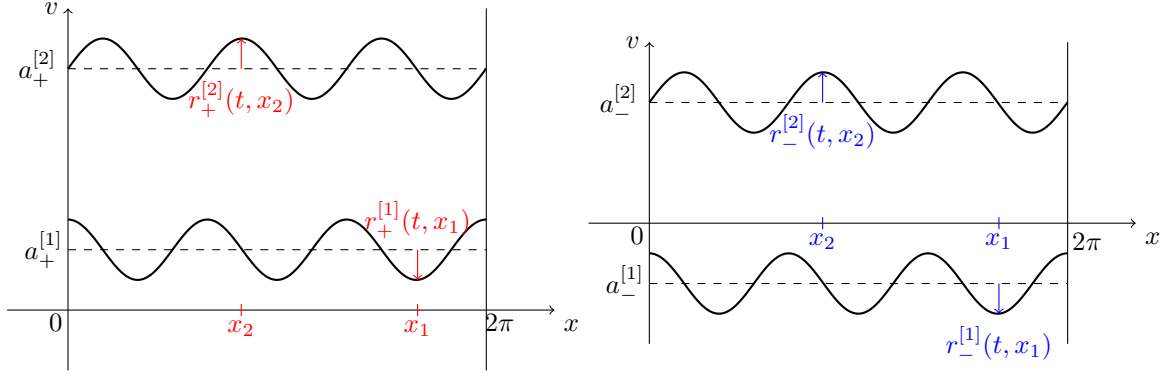
	
	Plugging \eqref{vel pot} and \eqref{electric potential in terms of r} into \eqref{IEP:eq}, we end up with the following system
	\begin{equation}\label{system principal}
		\begin{cases}
			\partial_{t}r_{+}^{[1]}=-\partial_{x}\Big(\tfrac{1}{2}\big( a_+^{[1]}+r_{+}^{[1]}\big)^2-\partial_{xx}^{-1}\big(r_+^{[2]}-r_+^{[1]}\big)+\partial_{xx}^{-1}\big(r_-^{[2]}-r_-^{[1]}\big)\Big),\vspace{0.2cm}\\
			\partial_{t}r_{+}^{[2]}=-\partial_{x}\Big(\tfrac{1}{2}\big(a_+^{[2]}+ r_{+}^{[2]}\big)^2-\partial_{xx}^{-1}\big(r_+^{[2]}-r_+^{[1]}\big)+\partial_{xx}^{-1}\big(r_-^{[2]}-r_-^{[1]}\big)\Big),\vspace{0.2cm}\\
			\partial_{t}r_{-}^{[1]}=-\partial_{x}\Big(\tfrac{1}{2}\big( a_-^{[1]}+r_{-}^{[1]}\big)^2+\partial_{xx}^{-1}\big(r_+^{[2]}-r_+^{[1]}\big)-\partial_{xx}^{-1}\big(r_-^{[2]}-r_-^{[1]}\big)\Big),\vspace{0.2cm}\\
			\partial_{t}r_{-}^{[2]}=-\partial_{x}\Big(\tfrac{1}{2}\big(a_-^{[2]}+ r_{-}^{[2]}\big)^2+\partial_{xx}^{-1}\big(r_+^{[2]}-r_+^{[1]}\big)-\partial_{xx}^{-1}\big(r_-^{[2]}-r_-^{[1]}\big)\Big),
		\end{cases}
	\end{equation}
	which can be recast in the following form
	\begin{equation}\label{system rpmi developpe}
		\begin{cases}
			\partial_{t}r_{+}^{[1]}+\big( a_+^{[1]}+r_{+}^{[1]}\big)\partial_{x}r_+^{[1]}-\partial_{x}^{-1}\big(r_+^{[2]}-r_+^{[1]}\big)+\partial_{x}^{-1}\big(r_-^{[2]}-r_-^{[1]}\big)=0,\vspace{0.2cm}\\
			\partial_{t}r_{+}^{[2]}+\big(a_+^{[2]}+ r_{+}^{[2]}\big)\partial_xr_+^{[2]}-\partial_{x}^{-1}\big(r_+^{[2]}-r_+^{[1]}\big)+\partial_{x}^{-1}\big(r_-^{[2]}-r_-^{[1]}\big)=0,\vspace{0.2cm}\\
			\partial_{t}r_{-}^{[1]}+\big( a_-^{[1]}+r_{-}^{[1]}\big)\partial_{x}r_{-}^{[1]}+\partial_{x}^{-1}\big(r_+^{[2]}-r_+^{[1]}\big)-\partial_{x}^{-1}\big(r_-^{[2]}-r_-^{[1]}\big)=0,\vspace{0.2cm}\\
			\partial_{t}r_{-}^{[2]}+\big(a_-^{[2]}+ r_{-}^{[2]}\big)\partial_{x}r_{-}^{[2]}+\partial_{x}^{-1}\big(r_+^{[2]}-r_+^{[1]}\big)-\partial_{x}^{-1}\big(r_-^{[2]}-r_-^{[1]}\big)=0,
		\end{cases}
	\end{equation}
	where the operator $\partial_x^{-1}$ is defined on the cosine/sine basis by
	\begin{equation*}
		\forall j\in\mathbb{N}^*,\qquad\partial_{x}^{-1}\cos(jx)\triangleq\frac{\sin(jx)}{j},\qquad\partial_{x}^{-1}\sin(jx)\triangleq-\frac{\cos(jx)}{j}\cdot
	\end{equation*}
	The set of equations \eqref{system principal} (or \eqref{system rpmi developpe}) is a system of four coupled quasilinear transport equations. In addition the coupling is linear and of order $-1$ in the unknowns. 
	\subsection{Main results}
	Before presenting our main result, we mention the previous work \cite{R23}. There, the ions were assumed to have significant inertia in order to consider them as a neutralizing uniform background field. The situation was like if the density $f_+$ had an electron sheet form
	$$f_+(t,x,v)=\delta_0(v)$$
	and therefore the system \eqref{Two component VP} was reduced to the single component Vlasov-Poisson equation given by the second equation in \eqref{Two component VP}. In that case, the neutrality condition is
	$$\int_{\mathbb{T}}\int_{\mathbb{R}}f_-(t,x,v)dvdx=1$$
	and electron layers are solutions in the form
	$$f_-(t,x,v)=\frac{1}{|S_0|}\mathbf{1}_{S_t},\qquad S_t=\Big\{(x,v)\in\mathbb{T}\times\mathbb{R}\quad\textnormal{s.t.}\quad v_-(t,x)<v<v_+(t,x)\Big\}.$$
	The flat strips are homogeneous trivial solutions and the bifurcation analysis near $(v_-,v_+)\equiv(a,b)$ has been performed in \cite{R23} finding for any fixed symmetry $\mathbf{m}\in\mathbb{N}^*$ two velocity bifurcation points given by (adapted to the $2\pi$-periodic setting)
	$$c_{\mathbf{m}}^{\pm}(a,b)=\frac{a+b}{2}\pm\sqrt{\frac{\left(\frac{b-a}{2}\right)^2\mathbf{m}^2+1}{\mathbf{m}^2}}\cdot$$
	In addition, the local bifurcation diagram has a "hyperbolic" structure, namely subcritical pitchfork bifurcation at $c_{\mathbf{m}}^-(a,b)$ and supercritical pitchfork bifurcation at $c_{\mathbf{m}}^+(a,b)$.\\
	
	In the two-component case of interest in this paper, the positive ions are supposed to have small enough inertia so that their motion is on a time scale comparable with that of the electrons in the plasma. Let us give the following definition
	\begin{defin}\textbf{(I/E-states or Ions-Electrons-states)}\\
		Let $\mathbf{m}\in\mathbb{N}^*$ and $c\in\mathbb{R}.$ we say that an ions-electrons layer is
		\begin{enumerate}[label=\textbullet]
			\item $\mathbf{m}$-symmetric if and only if
			$$\forall k\in\{1,2\},\quad\forall\kappa\in\{-,+\},\quad\forall t\geqslant0,\quad\forall x\in\mathbb{T},\quad r_{\kappa}^{[k]}\big(t,x+\tfrac{2\pi}{\mathbf{m}}\big)=r_{\kappa}^{[k]}(t,x).$$
			\item an I(ons)/E(lectrons)-state with velocity $c$ if for any $k\in\{1,2\}$ and any $\kappa\in\{-,+\}$, there exists $\check{r}_{\kappa}^{[k]}\in L^2(\mathbb{T})$ such that
			$$\forall t\geqslant0,\quad\forall x\in\mathbb{T},\quad r_{\kappa}^{[k]}(t,x)=\check{r}_{\kappa}^{[k]}(x-ct).$$
		\end{enumerate}
	\end{defin}
	
	Our first main result deals with the construction of local bifurcation curves of Ions-Electrons-states and reads as follows.
	
	\begin{theo}\label{thm:IEstates}\textbf{(Small amplitude Ions-Electrons-states)}\\
		Let $a\triangleq\left(a_+^{[1]},a_+^{[2]},a_-^{[1]},a_-^{[2]}\right)\in\mathbb{R}^4$ with
		\begin{equation}\label{largeur}
			\Delta a\triangleq a_+^{[2]}-a_+^{[1]}=a_-^{[2]}-a_-^{[1]}>0.
		\end{equation}
		\begin{enumerate}[label=(\roman*)]
			\item (Generic case) Assume that the components of $a$ are pairwise distinct, namely
			$$\left|\left\{a_+^{[1]},a_+^{[2]},a_-^{[1]},a_-^{[2]}\right\}\right|=4.$$
			Then, there exists $\underline{\mathbf{m}}\triangleq\underline{\mathbf{m}}\left(a_+^{[1]},a_+^{[2]},a_-^{[1]},a_-^{[2]}\right)\in\mathbb{N}^*$ such that for any $\mathbf{m}\in\mathbb{N}^*$ with $\mathbf{m}\geqslant\underline{\mathbf{m}},$ there exist four local bifurcation curves of $\mathbf{m}$-symmetric Ions-Electrons-states, emerging from the trivial solution $S_{\textnormal{flat}}(a)$ and taking the form
			$$\mathscr{C}_{\textnormal{\tiny{local}}}^{[k,\kappa]}(a)\triangleq\left\{\left(c_{\mathbf{m}}^{[k,\kappa]}(\mathtt{s},a),\check{r}_{\mathbf{m}}^{[k,\kappa]}(\mathtt{s},a)\right),\qquad|\mathtt{s}|\leqslant\delta\right\},\qquad \delta>0,\qquad k\in\{1,2\},\qquad\kappa\in\{-,\kappa\}$$
			with the asymptotics
			$$c_{\mathbf{m}}^{[k,\kappa]}(a)\triangleq c_{\mathbf{m}}^{[k,\kappa]}(0,a)\underset{\mathbf{m}\to\infty}{\longrightarrow}a_{\kappa}^{[k]}$$
			and
			\begin{equation}\label{asymptotic-rkpm}
				\forall x\in\mathbb{T},\quad r_{\mathbf{m}}^{[k,\kappa]}(\mathtt{s},a)(x)\underset{\mathtt{s}\to0}{=}\mathtt{s}\begin{pmatrix}
					\left(a_+^{[1]}-c_{\mathbf{m}}^{[k,\kappa]}(a)\right)^{-1}\\
					\left(a_+^{[2]}-c_{\mathbf{m}}^{[k,\kappa]}(a)\right)^{-1}\\
					-\left(a_-^{[1]}-c_{\mathbf{m}}^{[k,\kappa]}(a)\right)^{-1}\\
					-\left(a_-^{[2]}-c_{\mathbf{m}}^{[k,\kappa]}(a)\right)^{-1}
				\end{pmatrix}\cos(\mathbf{m}x)+O(\mathtt{s}^2).
			\end{equation}
			In addition,
			$$\textnormal{the bifurcation is pitchfork }\begin{cases}
				\textnormal{supercritical,} & \textnormal{if }c_{\mathbf{m}}^{[k,\kappa]}(a)>a_{\kappa}^{[k]},\vspace{0.2cm}\\
				\textnormal{subcritical,} & \textnormal{if }c_{\mathbf{m}}^{[k,\kappa]}(a)<a_{\kappa}^{[k]}.
			\end{cases}$$
			\item (Symmetric case) Assume that
			\begin{equation*}
				a_-^{[1]}=a_+^{[1]}\triangleq a_1\qquad\textnormal{and}\qquad a_-^{[2]}=a_+^{[2]}\triangleq a_2.
			\end{equation*}
			Then, there exist two local curves of $\mathbf{m}$-symmetric Ions-Electrons-states, emerging from the trivial solution $S_{\textnormal{flat}}(a)$, and taking the form
			$$\mathscr{C}_{\textnormal{\tiny{local}}}^{\pm}(a)\triangleq\left\{\left(c_{\mathbf{m}}^{\pm}(\mathtt{s},a),\check{r}_{\mathbf{m}}^{\pm}(\mathtt{s},a)\right),\qquad|\mathtt{s}|\leqslant\delta\right\},\qquad \delta>0,$$
			with 
			\begin{equation}\label{def:cpm-thm}
				c_{\mathbf{m}}^{\pm}(0,a)=c_{\mathbf{m}}^{\pm}(a_1,a_2)\triangleq\frac{a_1+a_2}{2}\pm\frac{1}{2}\sqrt{(a_2-a_1)^2+\frac{8\Delta a}{\mathbf{m}^2}}.
			\end{equation}
			The deformation function $r_{\mathbf{m}}^{\pm}$ admits the same asymptotic \eqref{asymptotic-rkpm} with $c_{\mathbf{m}}^{[k,\kappa]}(a)$ replaced by $c_{\mathbf{m}}^{\pm}(a_1,a_2).$
			Both bifurcations are of pitchfork-type and, asymptotically in $\mathbf{m}$, the bifurcation diagram admits (locally close to the trivial line) a "hyperbolic" structure as represented in the Figure \ref{figure hyperbolic sym}.
			\begin{figure}[!h]
				\begin{center}
					\begin{tikzpicture}
						\draw[->] (-5,0)--(5,0)
						node[below right] {$c$};
						\draw[>=latex, <->, black] (-1.3,-0.4)--(1.3,-0.4);
						\node at (0,-0.7) {$\Delta a$};
						\filldraw [black] (-1.5,0) circle (2pt);
						\filldraw [black] (1.5,0) circle (2pt);
						\node at (-1.5,-0.4) {$a_1$};
						\node at (1.5,-0.4) {$a_2$};
						\node at (-2.2,0.4) {$c_\mathbf{m}^{-}(a_1,a_2)$};
						\node at (2.2,0.4) {$c_\mathbf{m}^{+}(a_1,a_2)$};
						\filldraw [black] (-3,0) circle (2pt);
						\filldraw [black] (3,0) circle (2pt);
						\draw[domain=3:4,thick, black,samples=500] plot [variable=\t] (\t,{sqrt(\t-3)});
						\draw[domain=3:4,thick, black,samples=500] plot [variable=\t] (\t,{-sqrt(\t-3)});
						\draw[domain=-4:-3,thick, black,samples=500] plot [variable=\t] (\t,{sqrt(abs(3+\t))});
						\draw[domain=-4:-3,thick, black,samples=500] plot [variable=\t] (\t,{-sqrt(abs(3+\t))});
					\end{tikzpicture}
				\end{center}
				\caption{Local bifurcation diagram with "hyperbolic" structure in the symmetric case.}\label{figure hyperbolic sym}
			\end{figure}
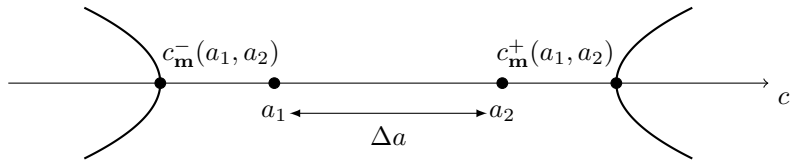
			\item (Successive case) Assume that
			\begin{equation*}
				a_{\kappa}^{[k]}=a_{-\kappa}^{[3-k]}\qquad\left(\textnormal{i.e. } a_-^{[1]}=a_+^{[2]}\,\,\textnormal{ or }\,\,a_+^{[1]}=a_-^{[2]}\right).
			\end{equation*}
			Then, there exist two local curves of $\mathbf{m}$-symmetric Ions-Electrons-states, emerging from the trivial solution $S_{\textnormal{flat}}(a)$, and taking the form
			$$\mathscr{C}_{\textnormal{\tiny{local}}}^{[k,\kappa,\pm]}(a)\triangleq\left\{\left(c_{\mathbf{m}}^{[k,\kappa,\pm]}(\mathtt{s},a),\check{r}_{\mathbf{m}}^{[k,\kappa,\pm]}(\mathtt{s},a)\right),\qquad|\mathtt{s}|\leqslant\delta\right\},\qquad \delta>0,$$
			with 
			$$c_{\mathbf{m}}^{[k,\kappa,\pm]}(0,a)=c_{\mathbf{m}}^{\pm}\left(a_{\kappa}^{[3-k]},a_{-\kappa}^{[k]}\right),$$
			with $c_{\mathbf{m}}^{\pm}$ as in \eqref{def:cpm-thm}.
			The deformation function $r_{\mathbf{m}}^{[k,\kappa,\pm]}$ admits the same asymptotic \eqref{asymptotic-rkpm} with $c_{\mathbf{m}}^{[k,\kappa]}(a)$ replaced by $c_{\mathbf{m}}^{\pm}\left(a_{\kappa}^{[3-k]},a_{-\kappa}^{[k]}\right).$
			Both bifurcations are of pitchfork-type and, asymptotically in $\mathbf{m}$, the bifurcation diagram admits (locally close to the trivial line) a "hyperbolic" structure as represented in the Figure \ref{figure hyperbolic suc}, where we denoted
			\begin{align*}
				\mathtt{m}_{k,\kappa}\triangleq\min\left(a_{\kappa}^{[3-k]},a_{-\kappa}^{[k]}\right)\qquad\textnormal{and}\qquad \mathtt{M}_{k,\kappa}\triangleq\max\left(a_{\kappa}^{[3-k]},a_{-\kappa}^{[k]}\right).
			\end{align*}
			
				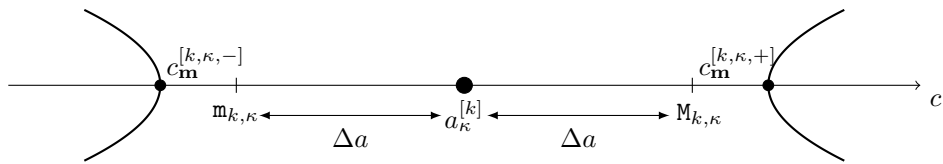
\begin{figure}[!h]
				\begin{center}
					\begin{tikzpicture}
						\draw[->] (-6,0)--(6,0)
						node[below right] {$c$};
						\draw[>=latex, <->, black] (-2.7,-0.4)--(-0.3,-0.4);
						\node at (-1.5,-0.7) {$\Delta a$};
						\draw[>=latex, <->, black] (2.7,-0.4)--(0.3,-0.4);
						\node at (1.5,-0.7) {$\Delta a$};
						\node at (0,-0.4) {$a_{\kappa}^{[k]}$};
						\draw[black] (-3,0.1)--(-3,-0.1);
						\node at (-3,-0.4) {$\mathtt{m}_{k,\kappa}$};
						\draw[black] (3,0.1)--(3,-0.1);
						\node at (3.1,-0.4) {$\mathtt{M}_{k,\kappa}$};
						\filldraw [black] (0,0) circle (3pt);
						\node at (-3.4,0.3) {$c_\mathbf{m}^{[k,\kappa,-]}$};
						\node at (3.6,0.3) {$c_\mathbf{m}^{[k,\kappa,+]}$};
						\filldraw [black] (-4,0) circle (2pt);
						\filldraw [black] (4,0) circle (2pt);
						\draw[domain=4:5,thick, black,samples=500] plot [variable=\t] (\t,{sqrt(\t-4)});
						\draw[domain=4:5,thick, black,samples=500] plot [variable=\t] (\t,{-sqrt(\t-4)});
						\draw[domain=-5:-4,thick, black,samples=500] plot [variable=\t] (\t,{sqrt(abs(4+\t))});
						\draw[domain=-5:-4,thick, black,samples=500] plot [variable=\t] (\t,{-sqrt(abs(4+\t))});
					\end{tikzpicture}
				\end{center}
				\caption{Local bifurcation diagram with "hyperbolic" structure in the successive case.}\label{figure hyperbolic suc}
			\end{figure}
		\end{enumerate}
	\end{theo}
	The proof of Theorem \ref{thm:IEstates} is performed in Section \ref{sec loc} and is based on the implementation of local bifurcation theory from simple eigenvalues \cite{CR71}, see also Theorem \ref{thm CR+S}. We follow the duality argument given in \cite{R22} for the scalar case and then extended to $2\times 2$ matricial situations \cite{CFLQ26,GHR23,HHRZ25,MRS24,R23}. In the present paper, the situation is more challenging since associated with $4\times 4$ matricial case. This makes difficult the explicit computation of the spectrum (doable only in the symmetric and successive cases) and forces an asymptotic analysis in the generic case. Actually, the proof given here should work the same in the situation with any finite number of interfaces both for the pure electronic and the two-species equations. As mentioned at the beginning of this subsection, the local "hyperbolic" shape of the bifurcation diagram was already observed in \cite{R23} for single component Vlasov-Poisson equation and seems reminiscent for these kinetic type models compared to fluid ones.\\
	
	Our second main result deals with the global extension of the curves provided in the previous theorem. Its proof is given in Section \ref{sec glo} and relies on the use of Theorem \ref{thm BT} from global bifurcation theory.
	
	\begin{theo}\label{thm:GB}\textbf{(Large amplitude Ions-Electrons-states)}
		All the bifurcations of Theorem \ref{thm:IEstates} are global in the Sobolev-analytic space $H^{s,\sigma}$ for $s>\tfrac{3}{2}$ and $\sigma>0.$ More precisely, 
		Let $a\triangleq\left(a_+^{[1]},a_+^{[2]},a_-^{[1]},a_-^{[2]}\right)\in\mathbb{R}^4$ and $\mathbf{m}\in\mathbb{N}^*$. Consider a local curve
		$$\mathscr{C}_{\textnormal{\tiny{local}}}^{\mathbf{m}}(a)=\left\{\left(\overline{c}_{\mathbf{m}}(\mathtt{s},a),\check{r}_{\mathbf{m}}(\mathtt{s},a)\right),\qquad|\mathtt{s}|\leqslant\delta\right\},\qquad \delta>0,$$
		being one of the local curves constructed in Theorem \ref{thm:IEstates}. Then, there exist a global curve
			$$\mathscr{C}_{\textnormal{\tiny{global}}}^{\mathbf{m}}(a)\triangleq\Big\{\left(\overline{c}_{\mathbf{m}}(\mathtt{s},a),\check{r}_{\mathbf{m}}(\mathtt{s},a)\right),\quad\mathtt{s}\in\mathbb{R}\Big\}$$
			corresponding to $\mathbf{m}$-symmetric E-states and extending the local curves $\mathscr{C}_{\textnormal{\tiny{local}}}^{\mathbf{m}}(a)$, i.e.
			$$\mathscr{C}_{\textnormal{\tiny{local}}}^{\mathbf{m}}(a)\subset\mathscr{C}_{\textnormal{\tiny{global}}}^{\mathbf{m}}(a).$$
			Moreover, the curve $\mathscr{C}_{\textnormal{\tiny{global}}}^{\mathbf{m}}(a)$ admits locally around each of its points a real-analytic reparametrization. In addition, one has the following alternatives
			\begin{itemize}
				\item [$(A1)$] (Loop) There exist $T_{\mathbf{m}}(a)>0$ such that
				$$\forall\mathtt{s}\in\mathbb{R},\quad \overline{c}_{\mathbf{m}}\big(\mathtt{s}+T_{\mathbf{m}}(a),a\big)=\overline{c}_{\mathbf{m}}(\mathtt{s},a)\qquad\textnormal{and}\qquad \check{r}_{\mathbf{m}}\big(\mathtt{s}+T_{\mathbf{m}}(a),a\big)=\check{r}_{\mathbf{m}}(\mathtt{s},a).$$
				\item [$(A2)$] Denoting
				$$\check{r}_{\mathbf{m}}(\mathtt{s},a)\triangleq\left(\check{r}_{\mathbf{m}}^{[1,+]}(\mathtt{s},a),\check{r}_{\mathbf{m}}^{[2,+]}(\mathtt{s},a),\check{r}_{\mathbf{m}}^{[1,-]}(\mathtt{s},a),\check{r}_{\mathbf{m}}^{[2,-]}(\mathtt{s},a)\right),$$
				one of the following limits occurs (possibly simultaneously)
				\begin{enumerate}[label=\textbullet]
					\item (Blow-up) $\displaystyle\lim_{\mathtt{s}\to\pm\infty}\frac{1}{1+\left|\overline{c}_{\mathbf{m}}(\mathtt{s},a)\right|+\|\check{r}_{\mathbf{m}}(\mathtt{s},a)\|_{s,\sigma}}=0.$
					\item (Collision of the boundaries) $\displaystyle\lim_{\mathtt{s}\to\pm\infty}\min_{\kappa\in\{-,+\}}\min_{x\in\mathbb{T}}\left|\check{r}_{\mathbf{m}}^{[2,\kappa]}(\mathtt{s},a)(x)-\check{r}_{\mathbf{m}}^{[1,\kappa]}(\mathtt{s},a)(x)+a_{\kappa}^{[2]}-a_{\kappa}^{[1]}\right|=0.$
					\item (Degeneracy) $\displaystyle\lim_{\mathtt{s}\to\pm\infty}\min_{k\in\{1,2\}\atop\kappa\in\{-,+\}}\min_{x\in\mathbb{T}}\left|\check{r}_{\mathbf{m}}^{[k,\kappa]}(\mathtt{s},a)(x)+a_{\kappa}^{[k]}-\overline{c}_{\mathbf{m}}(\mathtt{s},a)\right|=0.$
				\end{enumerate}
			\end{itemize}
	\end{theo}
	
	\begin{remark}
		It is not clear how to get rid of one of the alternatives. Actually, the situation might be quite rich. One can immagine that the curves exchange their roles with possible formation an annihilation of loops as we pass from 
		\begin{center}
			generic $\to$ successive $\to$ generic $\to$ symmetric $\to$ generic $\to$ successive $\to$ generic
		\end{center}
		by sliding one patch accross the other.
	\end{remark}
	\noindent\textbf{Acknowledgements :} This work has been supported by the ERC STARTING GRANT 2021 "Hamiltonian Dynamics, Normal Forms and Water Waves" (HamDyWWa), Project Number: 101039762.
	\section{Reformulations}
	Here we provide two different ways of writing the system \eqref{system principal}. The first one focuses on the Hamiltonian nature of the equations. The associated Hamiltonian is the total energy (kinetic + potential). As for the second reformulation, we show that under a suitable affine change of variables, we recover the two-component Euler-Poisson system. As a consequence, our results translates into new solutions for the Euler-Poisson equations.
	\subsection{Hamiltonian structure}
	In this short section, our goal is to highlight the Hamiltonian nature of the system \eqref{system principal}. 
	
	\begin{prop}\label{prop Hamiltonian form}
		Let us consider the energy functional \begin{equation}\label{definition Hamiltonien}
				\mathcal{E}(t)\triangleq\mathcal{E}_{\textnormal{\tiny{kin}}}(t)+\mathcal{E}_{\textnormal{\tiny{pot}}}(t),
			\end{equation}
			where $\mathcal{E}_{\textnormal{\tiny{kin}}}(t)$ and $\mathcal{E}_{\textnormal{\tiny{kin}}}(t)$ are respectively the kinetic and the electrostatic potential energies defined by
			\begin{align*}
			\mathcal{E}_{\textnormal{\tiny{kin}}}(t)&\triangleq\int_{\mathbb{T}\times\mathbb{R}}\tfrac{v^2}{2}\big(f_+(t,x,v)+f_{-}(t,x,v)\big)dxdv,\\
			\mathcal{E}_{\textnormal{\tiny{pot}}}(t)&\triangleq-\tfrac{1}{2}\int_{\mathbb{T}}\left(\int_{\mathbb{R}}\big(f_+(t,x,v)-f_-(t,x,v)\big)dv\right)\boldsymbol{\varphi}(t,x)dx.
			\end{align*}
			Then, the system \eqref{system principal} is Hamiltonian in the sense that it can be written in the form
			\begin{equation}\label{HAM 2compVP}
				\partial_{t}r=\mathcal{J}\nabla\mathcal{E}(r),\qquad r\triangleq\left(r_{+}^{[1]},r_+^{[2]},r_-^{[1]},r_-^{[2]}\right),\qquad\mathcal{J}\triangleq\begin{pmatrix}
					\partial_{x} & 0 & 0 & 0\\
					0 & -\partial_{x} & 0 & 0\\
					0 & 0 & \partial_x & 0\\
					0 & 0 & 0 & -\partial_x
				\end{pmatrix},
			\end{equation}
			where $\nabla\triangleq\left(\nabla_{r_+^{[1]}},\nabla_{r_+^{[2]}},\nabla_{r_-^{[1]}},\nabla_{r_-^{[2]}}\right)$ denotes the $\mathbf{L}^2(\mathbb{T})\triangleq \big(L^2(\mathbb{T})\big)^4$ gradient associated with the scalar product
			\begin{align*}
				&\left\langle \left(u_+^{[1]},u_+^{[2]},u_-^{[1]},u_-^{[2]}\right),\left(v_+^{[1]},v_+^{[2]},v_-^{[1]},v_-^{[2]}\right)\right\rangle_{\mathbf{L}^2(\mathbb{T})}\\
				&\triangleq\int_{\mathbb{T}}\left(u_+^{[1]}(x)v_+^{[1]}(x)+u_+^{[2]}(x)v_+^{[2]}(x)+u_-^{[1]}(x)v_-^{[1]}(x)+u_-^{[2]}(x)v_-^{[2]}(x)\right)dx.
			\end{align*}
		\end{prop} 
		\begin{proof}
			The kinetic energy writes
			$$\mathcal{E}_{\textnormal{\tiny{kin}}}(r)(t)=\int_{\mathbb{T}}\left(\int_{a_+^{[1]}+r_+^{[1]}(t,x)}^{a_+^{[2]}+r_+^{[2]}(t,x)}\tfrac{v^2}{2}dv+\int_{a_-^{[1]}+r_-^{[1]}(t,x)}^{a_-^{[2]}+r_-^{[2]}(t,x)}\tfrac{v^2}{2}dv\right)dx.$$
			Let $k\in\{1,2\}.$ Differentiating with respect to $r_{\pm}^{[k]}$ in the direction $h_{\pm}^{[k]},$ we get
			\begin{align*}
				\left\langle\nabla_{r_{\pm}^{[k]}}\mathcal{E}_{\textnormal{\tiny{kin}}}(r)(t),h_{\pm}^{[k]}(t)\right\rangle_{L^2(\mathbb{T})}&=(-1)^{k}\int_{\mathbb{T}}\tfrac{\big(a_{\pm}^{[k]}+r_{\pm}^{[k]}(t,x)\big)^2}{2}h_{\pm}^{[k]}(t,x)dx.
			\end{align*}
			This implies
			\begin{equation}\label{L2 grad Ekin}
				\nabla_{r_{\pm}^{[k]}}\mathcal{E}_{\textnormal{\tiny{kin}}}(r)(t,x)=\frac{(-1)^{k}\big(a_{\pm}^{[k]}+r_{\pm}^{[k]}(t,x)\big)^2}{2}\cdot
			\end{equation}
			Our next step is to study the potential energy term. For this aim, we recall that
			\begin{align*}
				\int_{\mathbb{R}}\big(f_+(t,x,v)-f_-(t,x,v)\big)dv&=r_+^{[2]}(t,x)-r_+^{[1]}(t,x)-r_-^{[2]}(t,x)+r_-^{[1]}(t,x).
			\end{align*}
			Therefore, using \eqref{electric potential in terms of r}, we can write
			$$\mathcal{E}_{\textnormal{\tiny{pot}}}(r)(t)=\frac{-1}{2}\int_{\mathbb{T}}\left(r_+^{[2]}(t,x)-r_+^{[1]}(t,x)-r_-^{[2]}(t,x)+r_-^{[1]}(t,x)\right)\partial_{xx}^{-1}\left(r_+^{[2]}(t,x)-r_+^{[1]}(t,x)-r_-^{[2]}(t,x)+r_-^{[1]}(t,x)\right)dx.$$
			Differentiating with respect to $r_{\pm}^{[k]}$ in the direction $h_{\pm}^{[k]}$ and using the self-adjointness of $\partial_{xx}^{1}$ give
			\begin{align*}
				\left\langle\nabla_{r_{\pm}^{[k]}}\mathcal{E}_{\textnormal{\tiny{pot}}}(r)(t),h_{\pm}^{[k]}(t)\right\rangle_{L^2(\mathbb{T})}&=\frac{\pm(-1)^{k+1}}{2}\int_{\mathbb{T}}h_{\pm}^{[k]}(t,x)\partial_{xx}^{-1}\left(r_+^{[2]}(t,x)-r_+^{[1]}(t,x)-r_-^{[2]}(t,x)+r_-^{[1]}(t,x)\right)dx\\
				&\quad+\frac{\pm(-1)^{k+1}}{2}\int_{\mathbb{T}}\left(r_+^{[2]}(t,x)-r_+^{[1]}(t,x)-r_-^{[2]}(t,x)+r_-^{[1]}(t,x)\right)\partial_{xx}^{-1}h_{\pm}^{[k]}(t,x)dx\\
				&=\pm(-1)^{k+1}\int_{\mathbb{T}}h_{\pm}(t,x)\partial_{xx}^{-1}\left(r_+^{[2]}(t,x)-r_+^{[1]}(t,x)-r_-^{[2]}(t,x)+r_-^{[1]}(t,x)\right)dx,
			\end{align*}
			which implies in turn
			\begin{equation}\label{L2 grad Epot}
				\nabla_{r_{\pm}}\mathcal{E}_{\textnormal{\tiny{pot}}}(r)(t,x)=\pm(-1)^{k+1}\partial_{xx}^{-1}\left(r_+^{[2]}(t,x)-r_+^{[1]}(t,x)-r_-^{[2]}(t,x)+r_-^{[1]}(t,x)\right).
			\end{equation}
			Combining \eqref{definition Hamiltonien}, \eqref{L2 grad Ekin} and \eqref{L2 grad Epot}, we infer
			\begin{align*}
				\nabla_{r_{\pm}^{[k]}}\mathcal{E}(r)(t,x)=(-1)^{k}\left[\tfrac{1}{2}\left(a_{\pm}^{[k]}+r_{\pm}^{[k]}(t,x)\right)^2\mp\partial_{xx}^{-1}\left(r_+^{[2]}(t,x)-r_+^{[1]}(t,x)-r_-^{[2]}(t,x)+r_-^{[1]}(t,x)\right)\right].
			\end{align*}
			Comparing with \eqref{system principal}, we get \eqref{HAM 2compVP}. This achieves the proof of Proposition \ref{prop Hamiltonian form}.
		\end{proof}
		\subsection{Link with the two-component Euler-Poisson system}
		Our next purpose is to expose the relation between the system \eqref{system principal} for the evolution of the patch boundaries and the classical two-component Euler-Poisson system given by
		\begin{equation}\label{Euler-Poisson system}
			\begin{cases}
				\partial_{t}\rho_{\pm}+\partial_{x}\big(\rho_{\pm}u_{\pm}\big)=0,\\
				\partial_{t}\big(\rho_{\pm}u_{\pm}\big)+\partial_{x}\big(\rho_{\pm}u_{\pm}^2\big)+\partial_{x}P(\rho_{\pm})=\pm\rho_{\pm}\partial_{x}\phi,\\
				\partial_{xx}^2\phi=4\pi e(\rho_+-\rho_-).
			\end{cases}
		\end{equation}
		The system \eqref{Euler-Poisson system} describes the dynamics of ions and electrons through a self-consistent electric field. Here $1+\rho$, $u$ and $\partial_{x}\phi$ represent the electron density, the electron velocity and the self-consistent electric field, respectively. The thermal pressure of electrons $P(\rho)$ is often assumed to follow a polytropic $\gamma$ law
		\begin{equation}\label{polytopic law}
			P(\rho)=T\rho^{\gamma},\qquad T\in\mathbb{R},\qquad\gamma\geqslant1.
		\end{equation}
		We refer the reader to \cite{CDMS95,CDMS96,NRS23} for some literature about traveling waves for the Euler-Poisson system.\\
		
		We consider the symmetric situation
		$$a_+^{[2]}=a_-^{[2]}=a>0\qquad\textnormal{and}\qquad a_+^{[1]}=a_-^{[1]}=-a.$$
		Then, we define 
		\begin{equation}\label{new-var:EP}
			\rho_{\pm}\triangleq a+\frac{r_{\pm}^{[2]}-r_{\pm}^{[1]}}{2},\qquad u_{\pm}\triangleq\frac{r_{\pm}^{[2]}+r_{\pm}^{[1]}}{2}\cdot
		\end{equation}
		The transformation \eqref{new-var:EP} is affine in $r$. In particular, traveling waves for \eqref{system principal} become traveling waves for \eqref{Euler-Poisson system} (and conversely).
		Then, from \eqref{system principal}, we get
		\begin{align*}
			\partial_{t}\rho_{\pm}&=\frac{1}{2}\left(\partial_tr_{\pm}^{[2]}-\partial_tr_{\pm}^{[1]}\right)\\
			&=-\frac{1}{4}\partial_{x}\left(\left(a+ r_{\pm}^{[2]}\right)^2-\left(-a+r_{\pm}^{[1]}\right)^2\right)\\
			&=-\frac{1}{4}\partial_{x}\left(\left(2a+r_{\pm}^{[2]}-r_{\pm}^{[1]}\right)\left(r_{\pm}^{[2]}+r_{\pm}^{[1]}\right)\right)\\
			&=-\partial_{x}(\rho_{\pm}u_{\pm}),
		\end{align*}
		which corresponds to the first equation in \eqref{Euler-Poisson system}. Besides, \eqref{system principal} also implies
		\begin{align*}
			\partial_{t}u_{\pm}&=\frac{1}{2}\left(\partial_{t}r_{\pm}^{[2]}+\partial_{t}r_{\pm}^{[1]}\right)\\
			&=-\frac{1}{4}\partial_{x}\left(\left(a+r_{\pm}^{[2]}\right)^2+\left(-a+r_{\pm}^{[1]}\right)^2\right)\pm\partial_{x}^{-1}\left(r_{+}^{[2]}-r_{+}^{[1]}-r_-^{[2]}+r_{-}^{[1]}\right)\\
			&=-\tfrac{1}{2}\partial_{x}\big(\rho_{\pm}^2+u_{\pm}^2\big)\pm2\partial_{x}^{-1}(\rho_+-\rho_-).
		\end{align*}
		Therefore,
		\begin{align*}
			\partial_{t}(\rho_{\pm}u_{\pm})&=u_{\pm}\partial_{t}\rho_{\pm}+\rho_{\pm}\partial_{t}u_{\pm}\\
			&=-u_{\pm}\partial_{x}(\rho_{\pm}u_{\pm})+\rho_{\pm}\Big[-\tfrac{1}{2}\partial_{x}\big(\rho_{\pm}^2+u_{\pm}^2\big)\pm2\partial_{x}^{-1}(\rho_+-\rho_-)\Big]\\
			&=-u_{\pm}^2\partial_{x}\rho_{\pm}-2\rho_{\pm}u_{\pm}\partial_{x}u_{\pm}-\rho_{\pm}^2\partial_{x}\rho_{\pm}+\rho_{\pm}\partial_{x}^{-1}(\rho_+-\rho_-).
		\end{align*}
		Also,
		\begin{align*}
			\partial_{x}\big(\rho_{\pm}u_{\pm}^2\big)&=u_{\pm}^2\partial_{x}\rho_{\pm}+2\rho_{\pm}u_{\pm}\partial_{x}u_{\pm}.
		\end{align*}
		Combining the foregoing calculations yields
		\begin{align*}
			\partial_{t}(\rho_{\pm}u_{\pm})+\partial_{x}\big(\rho_{\pm}u_{\pm}^2\big)\mp2\rho_{\pm}\partial_{x}^{-1}(\rho_{+}-\rho_{-})&=-\tfrac{1}{3}\partial_{x}\big(\rho_{\pm}^3\big).
		\end{align*}
		This corresponds to the second (and third) equation in \eqref{Euler-Poisson system} with the charge renormalization
		$$e=\frac{1}{2\pi}$$
		and a pressure term in the power law form
		\begin{equation}\label{cubic pressure law}
			P(\rho)=\frac{\rho^3}{3},
		\end{equation}
		that is \eqref{polytopic law} with
		$$T=\frac{1}{3}\qquad\textnormal{and}\qquad\gamma=3.$$
		Since the change of unknowns \eqref{new-var:EP} is affine we can state the following theorem whose proof is a direct consequence of Theorem \ref{thm:IEstates}-$(ii)$ and Theorem \ref{thm:GB}. We give here an informal statement, but the interested reader can perform the change of variables \eqref{new-var:EP} and obtain for instance the local asymptotic expansions.
		\begin{theo} Let $a>0$ and $\mathbf{m}\in\mathbb{N}^*.$ The two component Euler Poisson system with cubic pressure law \eqref{cubic pressure law} admits two global curves of solutions corresponding to small and large amplitude $\mathbf{m}$-symmetric traveling periodic waves 
			$$(\rho_{\pm},u_{\pm})(t,x)=(\check{\rho}_{\pm},\check{u}_{\pm})(x-ct),\qquad\check{\rho}_{\pm},\check{u}_{\pm}\in L^{2}(\mathbb{T}),\qquad(\check{\rho}_{\pm},\check{u}_{\pm})\left(x+\tfrac{2\pi}{\mathbf{m}}\right)=(\check{\rho}_{\pm},\check{u}_{\pm})(x),$$
			bifurcating from the trivial state $(\rho_{\pm},u_{\pm})\equiv(a,0)$ at the speeds
			$$c_{\mathbf{m}}^{\pm}(a)=\pm a\sqrt{1+\frac{2}{a\mathbf{m}^2}}\cdot$$
		\end{theo}

	\section{Small amplitude solutions}\label{sec loc}
	This section is devoted to the proof of Theorem \ref{thm:IEstates}. For this aim, we check the various hypothesis of Theorem~\ref{thm CR+S}. First, we reformulate the problem as the search of zeros for a time independant functional. Then, we introduce the functional framework and study the linearized operator at the trivial flat solutions. At last, we expose the asymptotic quantitative description of the local bifurcation diagram. \\
	
	We shall look for solutions of \eqref{system rpmi developpe} in the form
	$$r_{\pm}^{[k]}(t,x)=\check{r}_{\pm}^{[k]}(x-ct), \qquad c\in\mathbb{R},\qquad\check{r}_{\pm}^{[k]}\in L^{2}(\mathbb{T}),\qquad k\in\{1,2\}.$$
	In what follows, we denote
	$$a\triangleq\left(a_{+}^{[1]},a_{+}^{[2]},a_{-}^{[1]},a_{-}^{[2]}\right)\qquad\textnormal{and}\qquad\check{r}\triangleq\left(\check{r}_{+}^{[1]},\check{r}_{+}^{[2]},\check{r}_{-}^{[1]},\check{r}_{-}^{[2]}\right).$$
	With this ansatz, the system \eqref{system rpmi developpe} becomes
	\begin{equation}\label{bif-funct}
		\begin{array}{l}
			\forall x\in\mathbb{T},\quad F(a,c,\check{r})(x)=0,\qquad F\triangleq\left(F_+^{[1]}\,,\,F_+^{[2]}\,,\,F_-^{[1]}\,,\,F_-^{[2]}\right),\vspace{0.2cm}\\
			F_{\pm}^{[k]}(a,c,\check{r})(x)\triangleq \left(\check{r}_{\pm}^{[k]}(x)+a_{\pm}^{[k]}-c\right)\partial_{x}\check{r}_{\pm}^{[k]}(x)\mp\partial_{x}^{-1}\left(\check{r}_{+}^{[2]}(x)-\check{r}_{+}^{[1]}(x)-\check{r}_{-}^{[2]}(x)+\check{r}_{-}^{[1]}(x)\right).
		\end{array}
	\end{equation}
	We define the admissible sets
	\begin{align*}
		\mathcal{A}&\triangleq\Big\{(a_1,a_2,a_3,a_4)\in\mathbb{R}^4\quad\textnormal{s.t.}\quad a_2-a_1=a_4-a_3>0\Big\},\\
		\mathcal{A}_{\textnormal{sym}}&\triangleq\Big\{(a_1,a_2,a_3,a_4)\in\mathcal{A}\quad\textnormal{s.t.}\quad a_1=a_3\quad\textnormal{and}\quad a_2=a_4\Big\},\\
		\mathcal{A}_{\textnormal{suc}}&\triangleq\Big\{(a_1,a_2,a_3,a_4)\in\mathcal{A}\quad\textnormal{s.t.}\quad a_1=a_4\quad\textnormal{or}\quad a_2=a_3\Big\}.
	\end{align*}
	The set $\mathcal{A}$ capture the fact that the ions and electrons layers must have the same width, see \eqref{equalDeltas}. The set $\mathcal{A}_{\textnormal{sym}}$ describes the situation where both the ions and electrons flat strips coincide. Finally, the set $\mathcal{A}_{\textnormal{suc}}$ stands for the case where both flat layers are successive. One can easily check from \eqref{bif-funct} that
	\begin{equation}\label{triv}
		\forall a\in\mathcal{A},\quad\forall c\in\mathbb{R},\quad F(a,c,\mathbf{0})=0.
	\end{equation}
	The relation \eqref{triv} states that for any $c\in\mathbb{R}$ the flat ions/electrons strips are trivial solutions and corresponds to the hypothesis $(L1)$ of Theorem \ref{thm CR+S}. Now, we shall precise the functional framework. Given $s,\sigma\geqslant0$ and $\mathbf{m}\in\mathbb{N}^*$ we consider the following Sobolev-analytic function spaces
	\begin{align*}
		X_{\mathbf{m},\textnormal{\tiny{even}}}^{s,\sigma}&\triangleq\Big\{f\in L^{2}(\mathbb{T})\quad\textnormal{s.t.}\quad\forall x\in\mathbb{T},\quad f(x)=\sum_{j=1}^{\infty}f_{j}\cos(j\mathbf{m}x),\quad f_{j}\in\mathbb{R},\quad\sum_{j=1}^{\infty}j^{2s}e^{2\sigma j}f_{j}^2<\infty\Big\},\\
		Y_{\mathbf{m},\textnormal{\tiny{odd}}}^{s,\sigma}&\triangleq\Big\{g\in L^2(\mathbb{T})\quad\textnormal{s.t.}\quad\forall x\in\mathbb{T},\quad g(x)=\sum_{j=1}^{\infty}g_{j}\sin(j\mathbf{m}x),\quad g_{j}\in\mathbb{R},\quad\sum_{j=1}^{\infty}j^{2s}e^{2\sigma j}g_{j}^2<\infty\Big\}.
	\end{align*}
	Then, we consider the product spaces
	\begin{align*}
		\mathbb{X}_{\mathbf{m},\textnormal{\tiny{even}}}^{s,\sigma}&\triangleq X_{\mathbf{m},\textnormal{\tiny{even}}}^{s,\sigma}\times X_{\mathbf{m},\textnormal{\tiny{even}}}^{s,\sigma}\times X_{\mathbf{m},\textnormal{\tiny{even}}}^{s,\sigma}\times X_{\mathbf{m},\textnormal{\tiny{even}}}^{s,\sigma},\\
		\mathbb{Y}_{\mathbf{m},\textnormal{\tiny{odd}}}^{s,\sigma}&\triangleq Y_{\mathbf{m},\textnormal{\tiny{odd}}}^{s,\sigma}\times Y_{\mathbf{m},\textnormal{\tiny{odd}}}^{s,\sigma}\times Y_{\mathbf{m},\textnormal{\tiny{odd}}}^{s,\sigma}\times Y_{\mathbf{m},\textnormal{\tiny{odd}}}^{s,\sigma}.
	\end{align*}
	Both spaces are complete when endowed with the norm 
	$$\left\|\left(u_+^{[1]},u_+^{[2]},u_-^{[1]},u_-^{[2]}\right)\right\|_{s,\sigma}\triangleq\max_{k\in\{1,2\}\atop\kappa\in\{-,+\}}\left\|u_{\kappa}^{[k]}\right\|_{s,\sigma},\qquad\|u\|_{s,\sigma}\triangleq\left(\sum_{j=1}^{\infty}j^{2s}e^{2\sigma j}u_{j}^2\right)^{\frac{1}{2}}.$$
	Notice that these function spaces contain the zero average condition mathcing the condition \eqref{zero space average condition}. The interested reader is refered to \cite{ABZ22} for an introduction to the general Sobolev-analytic spaces and their properties. For later purposes, we insist on the fact that, for a fixed $\sigma\geqslant0$, the Sobolev-analytic scale $(H^{s,\sigma})_{s\geqslant 0}$ behaves like the classical Sobolev scale $(H^s)_{s\geqslant0}$. In particular, the classical estimates hold true (interpolation, product and composition laws, compact embeddings etc...). From the structure of the equations \eqref{bif-funct} (linear and quadratic terms), the classical formula
	\begin{equation*}
		\forall(u,v)\in\mathbb{R}^2,\quad\sin(u)\cos(v)=\tfrac{1}{2}\big(\sin(u+v)+\sin(u-v)\big),
	\end{equation*}
	implies that
	\begin{equation}\label{analF}
		\textnormal{for any }\sigma>0\textnormal{ and }s\geqslant 1,\textnormal{ the function }F:\mathcal{A}\times\mathbb{R}\times \mathbb{X}_{\mathbf{m},\textnormal{\tiny{even}}}^{s,\sigma}\rightarrow \mathbb{Y}_{\mathbf{m},\textnormal{\tiny{odd}}}^{s-1,\sigma}\textnormal{ is well-defined and analytic.}
	\end{equation}
	In the rest of this section, we fix the Sobolev-analytic regularity indices as $s\geqslant 1$ and $\sigma>0$. The linearized operator at $\check{r}=(0,0,0,0)\triangleq\mathbf{0}$ is
	\begin{equation*}
		\mathcal{L}_{c}\triangleq d_{\check{r}}F(a,c,\mathbf{0})=I_{0}+K_{0},
	\end{equation*}
	with
	\begin{equation}\label{lin op 2}
		\begin{aligned}
			I_{0}&\triangleq\begin{pmatrix}
				\left(a_+^{[1]}-c\right)\partial_{x} & 0 & 0 & 0\\
				0 & \left(a_+^{[2]}-c\right)\partial_{x} & 0 & 0\\
				0 & 0 & \left(a_-^{[1]}-c\right)\partial_{x} & 0\\
				0 & 0 & 0 & \left(a_-^{[2]}-c\right)\partial_{x}
			\end{pmatrix},\\
			K_{0}&\triangleq\begin{pmatrix}
				\partial_{x}^{-1} & -\partial_{x}^{-1} & -\partial_{x}^{-1} & \partial_{x}^{-1}\\
				\partial_{x}^{-1} & -\partial_{x}^{-1} & -\partial_{x}^{-1} & \partial_{x}^{-1}\\
				-\partial_{x}^{-1} & \partial_{x}^{-1} & \partial_{x}^{-1} & -\partial_{x}^{-1}\\
				-\partial_{x}^{-1} & \partial_{x}^{-1} & \partial_{x}^{-1} & -\partial_{x}^{-1}\\
			\end{pmatrix}.
		\end{aligned}
	\end{equation}
	Observe that if $c\not\in\left\{a_{+}^{[1]},a_{+}^{[2]},a_{-}^{[1]},a_{-}^{[2]}\right\},$ then $I_{0}:\mathbb{X}_{\mathbf{m},\textnormal{\tiny{even}}}^{s,\sigma}\rightarrow \mathbb{Y}_{\mathbf{m},\textnormal{\tiny{odd}}}^{s-1,\sigma}$ is an isomorphism. Moreover, by continuity of $K_{0}:\mathbb{X}_{\mathbf{m},\textnormal{\tiny{even}}}^{s,\sigma}\rightarrow \mathbb{Y}_{\mathbf{m}}^{s+1,\sigma}$, an application of Rellich's Theorem implies that $K_{0}:\mathbb{X}_{\mathbf{m},\textnormal{\tiny{even}}}^{s,\sigma}\rightarrow \mathbb{Y}_{\mathbf{m},\textnormal{\tiny{odd}}}^{s-1,\sigma}$ is a compact operator. As a consequence, 
	\begin{equation}\label{Fred}
		\begin{aligned}
			&\textnormal{for }c\not\in\left\{a_{+}^{[1]},a_{+}^{[2]},a_{-}^{[1]},a_{-}^{[2]}\right\},\,\sigma>0\textnormal{ and }s\geqslant 1,\\ \mathcal{L}_c:\mathbb{X}_{\mathbf{m},\textnormal{\tiny{even}}}^{s,\sigma}&\rightarrow \mathbb{Y}_{\mathbf{m},\textnormal{\tiny{odd}}}^{s-1,\sigma}\textnormal{ is a Fredholm operator with zero index.}
		\end{aligned}
	\end{equation}
	In addition, it is a Fourier multiplier whose action on an element $h=\left(h_+^{[1]},h_+^{[2]},h_-^{[1]},h_-^{[2]}\right)\in \mathbb{X}_{\mathbf{m},\textnormal{\tiny{even}}}^{s,\sigma}$ in the form
	$$h_{\pm}^{[k]}(x)=\sum_{j=1}^{\infty}h_{j}^{\pm,k}\cos(j\mathbf{m}x),\qquad h_{j}^{\pm}\in\mathbb{R},$$
	writes
	\begin{equation}\label{Fourier-rep}
		\mathcal{L}_{c}[h](x)=-\sum_{j=1}^{\infty}\frac{1}{j\mathbf{m}}M_{j\mathbf{m}}(a,c)\begin{pmatrix}
			h_{j}^{+,1}\\
			h_{j}^{+,2}\\
			h_{j}^{-,1}\\
			h_{j}^{-,2}
		\end{pmatrix}\sin(j\mathbf{m}x),
	\end{equation}
	where
	\begin{equation}\label{def:linmat}
		M_{j}(a,c)\triangleq \begin{pmatrix}
			j^2\left(a_+^{[1]}-c\right)-1 & 1 & 1 & -1\\
			-1 & j^2\left(a_+^{[2]}-c\right)+1 & 1 & -1\\
			1 & -1 & j^2\left(a_-^{[1]}-c\right)-1 & 1\\
			1 & -1 & -1 & j^2\left(a_-^{[2]}-c\right)+1
		\end{pmatrix}.
	\end{equation}
	In view of the application of Theorem \ref{thm CR+S}, we shall find, for a given $\mathbf{m}\in\mathbb{N}^*$, for which values of $c$ the matrix $M_{\mathbf{m}}(a,c)$ has a non-trivial kernel and hope that in that case it is of dimension $1$. Actually, this is what happens. More precisely, we have the following result.
	\begin{lem}\label{lem:singularmatrix}
		Let $a\in\mathcal{A}$ and $\mathbf{m}\in\mathbb{N}^*.$ The matrix $M_{\mathbf{m}}(a,c)$ is singular for 
		$$c\in\mathbb{E}_{\mathbf{m}}(a)\triangleq\left\{c_{\mathbf{m}}^{[1,+]}(a),c_{\mathbf{m}}^{[2,+]}(a),c_{\mathbf{m}}^{[1,-]}(a),c_{\mathbf{m}}^{[2,-]}(a)\right\}\subset\mathbb{C}.$$
		\begin{enumerate}[label=(\roman*)]
			\item Assume that $a\in\mathcal{A}_d\triangleq\mathcal{A}\setminus(\mathcal{A}_{\textnormal{sym}}\cup\mathcal{A}_{\textnormal{suc}})$ (i.e. the components of $a$ are pairwise distinct). Then, there exists $\underline{\mathbf{m}}\triangleq\underline{\mathbf{m}}\left(a_+^{[1]},a_+^{[2]},a_-^{[1]},a_-^{[2]}\right)\in\mathbb{N}^*$ such that 
			$$\mathbf{m}\geqslant\underline{\mathbf{m}}\qquad\Rightarrow\qquad\Big(\mathbb{E}_{\mathbf{m}}(a)\subset\mathbb{R}\setminus\left\{a_+^{[1]},a_+^{[2]},a_-^{[1]},a_-^{[2]}\right\}\quad\textnormal{and}\quad|\mathbb{E}_{\mathbf{m}}(a)|=4\Big).$$
			In addition, we have the asymptotics
			\begin{equation}\label{asymptotic}
				\forall k\in\{1,2\},\quad\forall\kappa\in\{-,+\},\quad c_{\mathbf{m}}^{[k,\kappa]}(a)\underset{\mathbf{m}\to\infty}{\longrightarrow}a_{\kappa}^{[k]}.
			\end{equation}
			\item Assume that $a\in\mathcal{A}_{\textnormal{sym}}$ and denote
			$$a_{+}^{[1]}=a_{-}^{[1]}\triangleq a_1\qquad\textnormal{and}\qquad a_{+}^{[2]}=a_{-}^{[2]}\triangleq a_2.$$
			Then, 
			$$\mathbb{E}_{\mathbf{m}}(a)=\Big\{a_1,a_2,c_{\mathbf{m}}^{-}(a_1,a_2),c_{\mathbf{m}}^{+}(a_1,a_2)\Big\}\subset\mathbb{R},\qquad|\mathbb{E}_{\mathbf{m}}(a)|=4,$$
			where we denote for $\alpha,\beta\in\mathbb{R},$
			\begin{equation}\label{def:cpmsym}
				c_{\mathbf{m}}^{\pm}(\alpha,\beta)\triangleq\frac{\alpha+\beta}{2}\pm\frac{1}{2}\sqrt{(\beta-\alpha)^2+\frac{8\Delta a}{\mathbf{m}^2}}\cdot
			\end{equation}
			We have the asymptotics
			\begin{equation}\label{asymptotic-cpm}
				c_{\mathbf{m}}^{-}(\alpha,\beta)\underset{\mathbf{m}\to\infty}{\nearrow}\min(\alpha,\beta)\qquad\textnormal{and}\qquad c_{\mathbf{m}}^{+}(\alpha,\beta)\underset{\mathbf{m}\to\infty}{\searrow}\max(\alpha,\beta).
			\end{equation}
			\item Assume $a\in\mathcal{A}_{\textnormal{suc}}$ more precisely that there exists $k\in\{1,2\}$ and $\kappa\in\{-,+\}$ such that
			$$a_{\kappa}^{[k]}=a_{-\kappa}^{[3-k]}.$$
			Then,
			$$\mathbb{E}_{\mathbf{m}}(a)=\left\{a_{k}^{[k]},c_{\mathbf{m}}^{-}\left(a_{-\kappa}^{[k]},a_{\kappa}^{[3-k]}\right),c_{\mathbf{m}}^{+}\left(a_{-\kappa}^{[k]},a_{\kappa}^{[3-k]}\right)\right\}\subset\mathbb{R},\qquad|\mathbb{E}_{\mathbf{m}}(a)|=3,$$
			with $c_{\mathbf{m}}^{\pm}$ defined trough \eqref{def:cpmsym}. The root $a_{k}^{[k]}$ is double.
			\item For $\overline{c}_{\mathbf{m}}(a)\in\big(\mathbb{E}_{\mathbf{m}}(a)\cap\mathbb{R}\big)\setminus\left\{a_+^{[1]},a_+^{[2]},a_-^{[1]},a_-^{[2]}\right\}$, one has the following one-dimensional kernel property
			\begin{equation}\label{dim1-kerM}
				\ker\Big(M_{\mathbf{m}}\big(a,\overline{c}_{\mathbf{m}}(a)\big)\Big)=\mathtt{span}(\mathtt{v}_0),\qquad\mathtt{v}_0\triangleq\begin{pmatrix}
					\left(a_+^{[1]}-\overline{c}_{\mathbf{m}}(a)\right)^{-1}\\
					\left(a_+^{[2]}-\overline{c}_{\mathbf{m}}(a)\right)^{-1}\\
					-\left(a_-^{[1]}-\overline{c}_{\mathbf{m}}(a)\right)^{-1}\\
					-\left(a_-^{[2]}-\overline{c}_{\mathbf{m}}(a)\right)^{-1}
				\end{pmatrix}.
			\end{equation}
			
		\end{enumerate}
	\end{lem}
	\begin{proof}
		Using standard algebraic computations, we find that the determinant of $M_{\mathbf{m}}(a,c)$ is
		\begin{equation}\label{deter}
			\Delta_{\mathbf{m}}(a,c)\triangleq\det\big(M_{\mathbf{m}}(a,c)\big)=\mathbf{m}^{8}\prod_{k\in\{1,2\}\atop\kappa\in\{-,+\}}\left(a_{\kappa}^{[k]}-c\right)+\mathbf{m}^{6}\sum_{k\in\{1,2\}\atop\kappa\in\{-,+\}}(-1)^k\prod_{p\in\{1,2\},\,\varsigma\in\{-,+\}\atop (p,\varsigma)\neq(k,\kappa)}\left(a_{\varsigma}^{[p]}-c\right)\in\mathbb{R}_4[c].
		\end{equation}
		It is a polynomial of degree $4$ in the variable $c$ so it has $4$ complex roots (counted with multiplicity) collected in the set $\mathbb{E}_{\mathbf{m}}(a)$. Observe that for $k\in\{1,2\}$ and $\kappa\in\{-,+\},$
		\begin{equation}\label{Delta at a}
			\Delta_{\mathbf{m}}\left(a,a_{\kappa}^{[k]}\right)=(-1)^{k}\mathbf{m}^6\prod_{p\in\{1,2\},\,\varsigma\in\{-,+\}\atop (p,\varsigma)\neq(k,\kappa)}\left(a_{\varsigma}^{[p]}-a_{\kappa}^{[k]}\right).
		\end{equation}
		Now, we need to distinguish the cases.\\
		$(i)$ Assume that $a\in\mathcal{A}_d,$ i.e. the components of $a$ are pairwise distinct. Hence, according to \eqref{Delta at a}, we get
		$$\Delta_{\mathbf{m}}\left(a,a_{\kappa}^{[k]}\right)\neq0,\qquad\textnormal{i.e.}\qquad\mathbb{E}_{\mathbf{m}}(a)\cap\left\{a_+^{[1]},a_+^{[2]},a_-^{[1]},a_-^{[2]}\right\}=\varnothing.$$
		Also, from \eqref{deter}, we can write
		$$\Delta_{\mathbf{m}}(a,c)=\mathbf{m}^{8}P_{\mathbf{m}}(c),\qquad P_{\mathbf{m}}\in \mathbb{R}_{4}[c]\,\,\text{unitary}.$$
		One redily has from \eqref{deter} that
		\begin{equation}\label{convergence-poly}
			P_{\mathbf{m}}\underset{\mathbf{m}\to\infty}{\longrightarrow}P,\qquad P(c)\triangleq\prod_{k\in\{1,2\}\atop\kappa\in\{-,+\}}\left(a_{\kappa}^{[k]}-c\right).
		\end{equation}
		Since $a\in\mathcal{A}_d$ then, the polynomial $P$ has simple real roots. We denote $\mathbb{R}_4^{(s)}[c]$ the subset of $\mathbb{R}_4[c]$ made of polynomials of degree $4$ with simple real roots. It is a classical fact (consequence of the Intermediate Value Theorem) that $\mathbb{R}_4^{(s)}[c]$ is an open subset of $\mathbb{R}_4[c]$.
			Since $P\in \mathbb{R}_4[c]$, we obtain from the convergence \eqref{convergence-poly} that for $\mathbf{m}$ large enough, $P_{\mathbf{m}}\in\mathbb{R}_4^{(s)}[c]$. This proves that, for $\mathbf{m}$ large enough,
			$$\mathbb{E}_{\mathbf{m}}(a)\subset\mathbb{R}\qquad\textnormal{and}\qquad|\mathbb{E}_{\mathbf{m}}(a)|=4.$$
			Finally, the asymptotics \eqref{asymptotic} follow immmediately from \eqref{convergence-poly} and the link roots-coefficients (up to a permutation of the roots).\\
			$(ii)$ Assume that $a\in\mathcal{A}_{\textnormal{sym}}$ and denote $a_+^{[1]}=a_-^{[1]}\triangleq a_1$ and $a_+^{[2]}=a_-^{[2]}\triangleq a_2.$ In that case, the determinant \eqref{deter} simplifies into
			\begin{align*}
				\Delta_{\mathbf{m}}(a,c)=\mathbf{m}^6\left(c-a_1\right)\left(c-a_2\right)Q_{\mathbf{m},a_1,a_2}(c),\qquad Q_{\mathbf{m},a_1,a_2}(c)\triangleq\mathbf{m}^2c^2-\mathbf{m}^2(a_1+a_2)c+\mathbf{m}^2 a_1a_2-2(a_2-a_1).
			\end{align*}
			Recall that
			$$a_2-a_1=\Delta a.$$
			Therefore,
			\begin{align*}
				Q_{\mathbf{m},a_1,a_2}(c)=0&\qquad\Leftrightarrow\qquad\mathbf{m}^2c^2-\mathbf{m}^2(a_1+a_2)c+\mathbf{m}^2 a_1a_2-2\Delta a=0\\
				&\qquad\Leftrightarrow\qquad c=\frac{a_1+a_2}{2}\pm\frac{1}{2\mathbf{m}^2}\sqrt{\mathbf{m}^4(a_1+a_2)^2-4\mathbf{m}^2\big(\mathbf{m}^2a_1a_2-2\Delta a\big)}\\
				&\qquad\Leftrightarrow\qquad c=\frac{a_1+a_2}{2}\pm\frac{1}{2}\sqrt{(a_2-a_1)^2+\frac{8\Delta a}{\mathbf{m}^2}},\\
				&\qquad\Leftrightarrow\qquad c\in\left\{c_{\mathbf{m}}^{-}(a_1,a_2),c_{\mathbf{m}}^{+}(a_1,a_2)\right\},
			\end{align*}
			with $c_{\mathbf{m}}^{\pm}(a_1,a_2)$ as in \eqref{def:cpmsym}.\\
			$(iii)$ Assume that $a\in\mathcal{A}_{\textnormal{suc}}$ and more precisely that there exists $k\in\{1,2\}$ and $\kappa\in\{-,+\}$ such that $a_{\kappa}^{[k]}=a_{-\kappa}^{[3-k]}.$ In that case, the determinant \eqref{deter} simplifies into
			$$\Delta_{\mathbf{m}}(a,c)=\mathbf{m}^6\left(c-a_{\kappa}^{[k]}\right)^2\widetilde{Q}_{\mathbf{m},k,\kappa,a}(c),$$
			with
			$$\widetilde{Q}_{\mathbf{m},k,\kappa,a}(c)\triangleq\mathbf{m}^2c^2-\mathbf{m}^2\left(a_{\kappa}^{[3-k]}+a_{-\kappa}^{[k]}\right)c+\mathbf{m}^2a_{\kappa}^{[3-k]}a_{-\kappa}^{[k]}-\left|a_{\kappa}^{[3-k]}-a_{-\kappa}^{[k]}\right|.$$
			Observe that by construction
			\begin{equation}\label{2largeurs}
				\left|a_{\kappa}^{[3-k]}-a_{-\kappa}^{[k]}\right|=2\Delta a
			\end{equation}
			so that
			$$\widetilde{Q}_{\mathbf{m},k,\kappa,a}=Q_{\mathbf{m},a_{\kappa}^{[3-k]},a_{-\kappa}^{[k]}}$$
			and
			$$\widetilde{Q}_{\mathbf{m},k,\kappa,a}(c)=0\qquad\Leftrightarrow\qquad c\in\left\{c_{\mathbf{m}}^{-}\left(a_{-\kappa}^{[k]},a_{\kappa}^{[3-k]}\right),c_{\mathbf{m}}^{+}\left(a_{-\kappa}^{[k]},a_{\kappa}^{[3-k]}\right)\right\}.$$
		$(iv)$ Let $X=(x_1,x_2,x_3,x_4)\in\mathbb{R}^4$. Then
		\begin{align*}
			X\in\ker\Big(M_{\mathbf{m}}\big(a,\overline{c}_{\mathbf{m}}(a)\big)\Big)&\quad\Leftrightarrow\quad \begin{aligned}[t]
				x_1-x_2-x_3+x_4&=\mathbf{m}^2\left(a_+^{[1]}-\overline{c}_{\mathbf{m}}(a)\right)x_1\\
				&=\mathbf{m}^2\left(a_+^{[2]}-\overline{c}_{\mathbf{m}}(a)\right)x_2\\
				&=-\mathbf{m}^2\left(a_-^{[1]}-\overline{c}_{\mathbf{m}}(a)\right)x_3\\
				&=-\mathbf{m}^2\left(a_-^{[2]}-\overline{c}_{\mathbf{m}}(a)\right)x_4
			\end{aligned}\\
			&\quad\Leftrightarrow\quad X\in\mathtt{span}(\mathtt{v}_0).
		\end{align*}
		This achieves the proof of Lemma \ref{lem:singularmatrix}.
	\end{proof}

The next proposition provides the remaining sufficient hypothesis for local bifurcation.
	
	\begin{prop}\label{prop:HypoCR}
		Let $a\in\mathcal{A}$, $\mathbf{m}\in\mathbb{N}^*$. Take $\overline{c}_{\mathbf{m}}(a)\in\big(\mathbb{E}_{\mathbf{m}}(a)\cap\mathbb{R}\big)\setminus\left\{a_+^{[1]},a_+^{[2]},a_-^{[1]},a_-^{[2]}\right\}$, that is we are in one of the following situations
		\begin{enumerate}[label=\textbullet]
			\item $a\in\mathcal{A}_d$, $\mathbf{m}\geqslant\underline{\mathbf{m}}$ and there exists $k\in\{1,2\}$ and $\kappa\in\{-,+\}$ such that 
			$$\overline{c}_{\mathbf{m}}(a)=c_{\mathbf{m}}^{[k,\kappa]}(a).$$ 
			\item $a\in\mathcal{A}_{\textnormal{sym}}$ with $a_{+}^{[1]}=a_{-}^{[1]}=a_1$, $a_{+}^{[2]}=a_-^{[2]}=a_2$ and 
			\begin{equation}\label{barc-sym}
				\overline{c}_{\mathbf{m}}(a)\in\left\{c_{\mathbf{m}}^{-}(a_1,a_2),c_{\mathbf{m}}^{+}(a_1,a_2)\right\}.
			\end{equation}
			\item $a\in\mathcal{A}_{\textnormal{suc}}$ with $a_{\kappa}^{[k]}=a_{-\kappa}^{[3-k]}$ for some $k\in\{1,2\}$ and $\kappa\in\{-,+\}$ and
			\begin{equation}\label{barc-suc}
				\overline{c}_{\mathbf{m}}(a)\in\left\{c_{\mathbf{m}}^{-}\left(a_{-\kappa}^{[k]},a_{\kappa}^{[3-k]}\right),c_{\mathbf{m}}^{+}\left(a_{-\kappa}^{[k]},a_{\kappa}^{[3-k]}\right)\right\}.
			\end{equation}
		\end{enumerate}
		Then, the following properties hold true concerning the linearized operator $\mathcal{L}_{\overline{c}_{\mathbf{m}}(a)}$.
		\begin{enumerate}[label=(\roman*)]
			\item The kernel is one dimensional. More precisely,
			\begin{equation}\label{dim1-kerL}
				\ker\left(\mathcal{L}_{\overline{c}_{\mathbf{m}}(a)}\right)=\mathtt{span}(\check{r}_0),\qquad\check{r}_0(x)\triangleq\mathtt{v}_0\cos(\mathbf{m}x), 
			\end{equation}
			where $\mathtt{v}_0$ is defined in \eqref{dim1-kerM}.
			\item The range is closed and of codimension one in $\mathbb{Y}_{\mathbf{m},\textnormal{\tiny{odd}}}^{s-1,\sigma}.$ More precisely,
			\begin{equation}\label{range:description}
				\mathcal{R}\left(\mathcal{L}_{\overline{c}_{\mathbf{m}}(a)}\right)=\mathtt{span}^{\perp}(y_0),\qquad y_0(x)\triangleq\mathtt{w}_0\sin(\mathbf{m}x),\qquad\mathtt{w}_0\triangleq\begin{pmatrix}
					\left(a_+^{[1]}-\overline{c}_{\mathbf{m}}(a)\right)^{-1}\\
					-\left(a_+^{[2]}-\overline{c}_{\mathbf{m}}(a)\right)^{-1}\\
					-\left(a_-^{[1]}-\overline{c}_{\mathbf{m}}(a)\right)^{-1}\\
					\left(a_-^{[2]}-\overline{c}_{\mathbf{m}}(a)\right)^{-1}
				\end{pmatrix},
			\end{equation}
			where the orthogonal is understood in the sense of the scalar product
			$$\langle y,\widetilde{y}\rangle=\sum_{j=1}^{\infty}\left(y_{j}^{[1,+]}\widetilde{y}_{j}^{[1,+]}+y_{j}^{[2,+]}\widetilde{y}_{j}^{[2,+]}+y_{j}^{[1,-]}\widetilde{y}_{j}^{[1,-]}+y_{j}^{[2,-]}\widetilde{y}_{j}^{[2,-]}\right).$$
			\item The transversality assumption is satisfied namely
			$$(\partial_{c}\mathcal{L}_{c})|_{c=\overline{c}_{\mathbf{m}}(a)}[\check{r}_0]\not\in\mathcal{R}\left(\mathcal{L}_{\overline{c}_{\mathbf{m}}(a)}\right).$$
		\end{enumerate}
	\end{prop}
	
	\begin{proof}
		$(i)$ Given $c\not\in\left\{a_{+}^{[1]},a_{+}^{[2]},a_{-}^{[1]},a_{-}^{[2]}\right\},$ we can factorize \eqref{deter} and obtain
				\begin{equation*}
						\Delta_{\mathbf{m}}(a,c)=\mathbf{m}^6\prod_{k\in\{1,2\}\atop\kappa\in\{-,+\}}\left(c-a_{\kappa}^{[k]}\right)\left[\mathbf{m}^2-\sum_{p\in\{1,2\}\atop\varsigma\in\{-,+\}}\frac{(-1)^p}{c-a_{\varsigma}^{[p]}}\right].
					\end{equation*}
				As a consequence, for $c\not\in\left\{a_{+}^{[1]},a_{+}^{[2]},a_{-}^{[1]},a_{-}^{[2]}\right\},$ we have
				\begin{equation}\label{det=0}
						\Delta_{\mathbf{m}}(a,c)=0\qquad\Leftrightarrow\qquad\mathbf{m}^2=\sum_{p\in\{1,2\}\atop\varsigma\in\{-,+\}}\frac{(-1)^p}{c-a_{\varsigma}^{[p]}}\cdot
					\end{equation}
				In view of \eqref{det=0}, one has
		$$\forall j\in\mathbb{N}\setminus\{0,1\},\quad\Delta_{j\mathbf{m}}\left(a,\overline{c}_{\mathbf{m}}(a)\right)\neq0.$$
		Therefore, together with \eqref{Fourier-rep} and \eqref{dim1-kerM}, this allows to conclude \eqref{dim1-kerL}.\\
		$(ii)$ Applying the orthogonal supplementary theorem to the finite dimensional subspace $\mathtt{span}(y_0)$ inside the pre-Hilbertian space $(\mathbb{Y}_{\mathbf{m},\textnormal{\tiny{odd}}}^{s-1,\sigma},\langle\cdot,\cdot\rangle)$, we find
		$$\mathbb{Y}_{\mathbf{m},\textnormal{\tiny{odd}}}^{s-1,\sigma}=\mathtt{span}(y_0)\overset{\perp}{\oplus}\mathtt{span}^{\perp}(y_0).$$
		This proves that $\mathtt{span}^{\perp}(y_0)$ is of codimension one in $\mathbb{Y}_{\mathbf{m},\textnormal{\tiny{odd}}}^{s-1,\sigma}$. Besides, the Fredholm property \eqref{Fred} together with the one dimensional kernel (point $(i)$) implies that also the range $\mathcal{R}\left(\mathcal{L}_{\overline{c}_{\mathbf{m}}(a)}\right)$ is (closed and) of codimension one in $\mathbb{Y}_{\mathbf{m},\textnormal{\tiny{odd}}}^{s-1,\sigma}$. Moreover, the transposed of the matrix of $M_{\mathbf{m}}(a,c)$ in \eqref{def:linmat} is
		$$M_{\mathbf{m}}^{\top}(a,c)=\begin{pmatrix}
			\mathbf{m}^2\left(a_+^{[1]}-c\right)-1 & -1 & 1 & 1\\
			1 & \mathbf{m}^2\left(a_+^{[2]}-c\right)+1 & -1 & -1\\
			1 & 1 & \mathbf{m}^2\left(a_-^{[1]}-c\right)-1 & -1\\
			-1 & -1 & 1 & \mathbf{m}^2\left(a_-^{[2]}-c\right)+1
		\end{pmatrix}.$$
		By construction of $\mathtt{w}_0$ in \eqref{range:description}, we have 
		\begin{equation}\label{ttw0}
			\mathtt{w}_0\in\ker\Big(M_{\mathbf{m}}^{\top}\big(a,\overline{c}_{\mathbf{m}}(a)\big)\Big).
		\end{equation}
		Thus, given any element $y$ of the range in the form
		$$y(x)=-\sum_{j=1}^{\infty}\frac{1}{j\mathbf{m}}M_{j\mathbf{m}}\big(a,\overline{c}_{\mathbf{m}}(a)\big)\vec{y}_j\sin(j\mathbf{m}x)\in\mathcal{R}\left(\mathcal{L}_{\overline{c}_{\mathbf{m}}(a)}\right),\qquad\vec{y}_j\triangleq\begin{pmatrix}
			y_j^{[1,+]}\vspace{0.1cm}\\
			y_j^{[2,+]}\vspace{0.1cm}\\
			y_j^{[1,-]}\vspace{0.1cm}\\
			y_j^{[2,-]}
		\end{pmatrix}\in\mathbb{R}^4,$$
		we have by \eqref{ttw0} that
		\begin{align*}
			\langle y,y_0\rangle&=-\frac{1}{\mathbf{m}}\left\langle M_{\mathbf{m}}\big(a,\overline{c}_{\mathbf{m}}(a)\big)\vec{y}_1,\mathtt{w}_0\right\rangle_{\mathbb{R}^4}\\
			&=-\frac{1}{\mathbf{m}}\left\langle\vec{y}_1,M_{\mathbf{m}}^{\top}\big(a,\overline{c}_{\mathbf{m}}(a)\big)\mathtt{w}_0\right\rangle_{\mathbb{R}^4}\\
			&=0.
		\end{align*}
		This proves the inclusion 
		$$\mathcal{R}\left(\mathcal{L}_{\overline{c}_{\mathbf{m}}(a)}\right)\subset\mathtt{span}^{\perp}(y_0).$$
		Together with the codimension one property for both subspaces, we can apply \cite[Lem. B.1]{R23} to get \eqref{range:description}.\\
		$(iii)$ Remark that
		$$\partial_{c}\mathcal{L}_c=\begin{pmatrix}
			-\partial_x & 0 & 0 & 0\\
			0 & -\partial_x & 0 & 0\\
			0 & 0 & -\partial_{x} & 0\\
			0 & 0 & 0 & -\partial_{x}
		\end{pmatrix}$$
		Hence,
		\begin{equation}\label{trans:test}
			\left\langle y_0,(\partial_{c}\mathcal{L}_{c})|_{c=\overline{c}_{\mathbf{m}}(a)}[\check{r}_0]\right\rangle=\mathbf{m}\langle\mathtt{v}_0,\mathtt{w}_0\rangle_{\mathbb{R}^4}=\mathbf{m}\sum_{p\in\{1,2\}\atop\varsigma\in\{-,+\}}^{4}(-1)^{p+1}\big(a_{\varsigma}^{[p]}-\overline{c}_{\mathbf{m}}(a)\big)^{-2}.
		\end{equation}
		Assume in view of a contradiction that $(\partial_{c}\mathcal{L}_{c})|_{c=\overline{c}_{\mathbf{m}}(a)}[\check{r}_0]\in\mathcal{R}\left(\mathcal{L}_{\overline{c}_{\mathbf{m}}(a)}\right).$ By virtue of \eqref{range:description} and \eqref{trans:test}, this is equivalent to
		\begin{equation}\label{impossible:trans}
			\big(a_{+}^{[1]}-\overline{c}_{\mathbf{m}}(a)\big)^{-2}+\big(a_{-}^{[1]}-\overline{c}_{\mathbf{m}}(a)\big)^{-2}=\big(a_{+}^{[2]}-\overline{c}_{\mathbf{m}}(a)\big)^{-2}+\big(a_{-}^{[2]}-\overline{c}_{\mathbf{m}}(a)\big)^{-2}.
		\end{equation}
		Now, we shall distinguish the cases.\\
		\texttt{Case $a\in\mathcal{A}_d$ and $\mathbf{m}$ large enough :} The asymptotic \eqref{asymptotic} and the the fact that the components of $a$ are distinct imply that the identity \eqref{impossible:trans} cannot hold since one of the members of the equation might blow up while the other tends to a finite number as $\mathbf{m}\to\infty.$ Contradiction.\\
		\texttt{Case $a\in\mathcal{A}_{\textnormal{sym}}$ and $\mathbf{m}\in\mathbb{N}^*$ :} The equation \eqref{impossible:trans} is reduced to
		\begin{equation}\label{reduced-trans}
			\big(a_1-c_{\mathbf{m}}^{\pm}(a_1,a_2)\big)^{2}=\big(a_2-c_{\mathbf{m}}^{\pm}(a_1,a_2)\big)^{2}.
		\end{equation}
		Using the explicit expression \eqref{def:cpmsym} and \eqref{largeur}, we find
		\begin{equation}\label{diffc}
			\begin{aligned}
				a_1-c_{\mathbf{m}}^{\pm}(a_1,a_2)&=-\frac{\Delta a}{2}\mp\frac{1}{2}\sqrt{(\Delta a)^2+\frac{8\Delta a}{\mathbf{m}^2}},\\
				a_2-c_{\mathbf{m}}^{\pm}(a_1,a_2)&=\frac{\Delta a}{2}\mp\frac{1}{2}\sqrt{(\Delta a)^2+\frac{8\Delta a}{\mathbf{m}^2}}\cdot
			\end{aligned}
		\end{equation}
		Inserting \eqref{diffc} into \eqref{reduced-trans}, we infer
		$$\Delta a\sqrt{(\Delta a)^2+\frac{8\Delta a}{\mathbf{m}^2}}=0.$$
		This enters in contradiction with \eqref{largeur}.\\
		\textit{Case $a\in\mathcal{A}_{\textnormal{suc}}$ and $\mathbf{m}\in\mathbb{N}^*$ :} We fix the notations $k\in\{1,2\}$ and $\kappa\in\{-,+\}$ such that $a_{\kappa}^{[k]}=a_{-\kappa}^{[3-k]}.$ The equation \eqref{impossible:trans} is reduced to
		\begin{equation}\label{reduced-trans-suc}
			\left(a_{\kappa}^{[3-k]}-c_{\mathbf{m}}^{\pm}\left(a_{\kappa}^{[3-k]},a_{-\kappa}^{[k]}\right)\right)^{-2}=\left(a_{-\kappa}^{[k]}-c_{\mathbf{m}}^{\pm}\left(a_{\kappa}^{[3-k]},a_{-\kappa}^{[k]}\right)\right)^{-2}.
		\end{equation} 
			Using the explicit expression \eqref{def:cpmsym} and \eqref{2largeurs}, we find
		\begin{equation}\label{diffsuc}
			\begin{aligned}
				a_{\kappa}^{[3-k]}-c_{\mathbf{m}}^{\pm}\left(a_{\kappa}^{[3-k]},a_{-\kappa}^{[k]}\right)&=(-1)^{k+1}\Delta a\mp\sqrt{(\Delta a)^2+\frac{2\Delta a}{\mathbf{m}^2}},\\
				a_{-\kappa}^{[k]}-c_{\mathbf{m}}^{\pm}\left(a_{\kappa}^{[3-k]},a_{-\kappa}^{[k]}\right)&=(-1)^{k}\Delta a\mp\sqrt{(\Delta a)^2+\frac{2\Delta a}{\mathbf{m}^2}}\cdot
			\end{aligned}
		\end{equation} 
		Inserting \eqref{diffsuc} into \eqref{reduced-trans-suc}, we infer
		$$\Delta a\sqrt{(\Delta a)^2+\frac{2\Delta a}{\mathbf{m}^2}}=0.$$
		This enters in contradiction with \eqref{largeur}. The proof of Proposition \ref{prop:HypoCR} is now complete.
	\end{proof}
	Putting together \eqref{triv}, \eqref{analF}, Proposition \ref{prop:HypoCR} and Theorem \ref{thm CR+S} provides the existence part of Theorem \ref{thm:IEstates}. We denote
	\begin{equation}\label{local curve}
		\mathscr{C}_{\textnormal{\tiny{local}}}^{\mathbf{m}}(a):\mathtt{s}\in(-\delta,\delta)\mapsto\Big(\overline{c}_{\mathbf{m}}(\mathtt{s},a),\check{r}_{\mathbf{m}}(\mathtt{s},a)\Big)\in\mathbb{R}\times\mathbb{X}_{\mathbf{m},\textnormal{\tiny{even}}}^{s,\sigma},\qquad\delta>0,
	\end{equation}
	the corresponding real-analytic local curve satisfying
	$$\overline{c}_{\mathbf{m}}(0,a)=\overline{c}_{\mathbf{m}}(a)\qquad\textnormal{and}\qquad\left.\frac{d}{d\mathtt{s}}\check{r}_{\mathbf{m}}(\mathtt{s},a)\right|_{\mathtt{s}=0}=\check{r}_0.$$
	We shall now study the pitchfork bifurcation property using the second part of Theorem \ref{thm CR+S}. Notice that for any $c\in\mathbb{R}$, $k\in\{1,2\}$, $\kappa\in\{-,+\}$, $h=\left(h_+^{[1]},h_+^{[2]},h_-^{[1]},h_{-}^{[2]}\right)\in\mathbb{X}_{\mathbf{m},\textnormal{\tiny{even}}}^{s,\sigma}$ and $\widetilde{h}=\left(\widetilde{h}_+^{[1]},\widetilde{h}_+^{[2]},\widetilde{h}_-^{[1]},\widetilde{h}_{-}^{[2]}\right)$, we have
	\begin{equation}\label{deriv2}
		d_{\check{r}}^2F_{\kappa}^{[k]}(a,c,\mathbf{0})\left[h,\widetilde{h}\right]=\partial_x\left(h_{\kappa}^{[k]}\widetilde{h}_{\kappa}^{[k]}\right)
	\end{equation}
and
\begin{equation}\label{deriv3}
	d_{\check{r}}^3F\equiv0.
\end{equation}
In what follows, we denote the Hessian
$$\mathcal{H}_{\overline{c}_{\mathbf{m}}(a)}\triangleq d_{\check{r}}^2F\left(a,\overline{c}_{\mathbf{m}}(a),\mathbf{0}\right).$$
By construction of $\check{r}_0$ in \eqref{dim1-kerL} and by virtue of \eqref{deriv2}, denoting
\begin{equation}\label{tildettw0}
	\widetilde{\mathtt{w}}_0\triangleq\begin{pmatrix}
		\left(a_+^{[1]}-\overline{c}_{\mathbf{m}}(a)\right)^{-2}\\
		\left(a_+^{[2]}-\overline{c}_{\mathbf{m}}(a)\right)^{-2}\\
		\left(a_-^{[1]}-\overline{c}_{\mathbf{m}}(a)\right)^{-2}\\
		\left(a_-^{[2]}-\overline{c}_{\mathbf{m}}(a)\right)^{-2}
	\end{pmatrix},
\end{equation}
we have
\begin{equation}\label{formula-Hess}
	\forall x\in\mathbb{T},\quad\mathcal{H}_{\overline{c}_{\mathbf{m}}(a)}[\check{r}_0,\check{r}_0](x)=\widetilde{\mathtt{w}}_0\,\partial_{x}\big(\cos^2(\mathbf{m}x)\big)=-\mathbf{m}\widetilde{\mathtt{w}}_0\sin(2\mathbf{m}x).
\end{equation}
The formula \eqref{formula-Hess} gives that the function $\mathcal{H}_{\overline{c}_{\mathbf{m}}(a)}[\check{r}_0,\check{r}_0]$ is localized on the Fourier mode $2\mathbf{m}.$ Thus, from \eqref{range:description}, we infer
$$\left\langle y_0,\mathcal{H}_{\overline{c}_{\mathbf{m}}(a)}[\check{r}_0,\check{r}_0]\right\rangle=0,\qquad\textnormal{i.e.}\qquad \mathcal{H}_{\overline{c}_{\mathbf{m}}(a)}[\check{r}_0,\check{r}_0]\in\mathcal{R}\left(\mathcal{L}_{\overline{c}_{\mathbf{m}}(a)}\right).$$
The Theorem \ref{thm CR+S} allows to conclude that the bifurcation is always of pitchfork type and we have, using in particular \eqref{deriv3}, that
\begin{equation}\label{local-coeff}
	\left.\frac{d}{d\mathtt{s}}\overline{c}_{\mathbf{m}}(\mathtt{s},a)\right|_{\mathtt{s}=0}=0\qquad\textnormal{and}\qquad\left.\frac{d^2}{d\mathtt{s}^2}\overline{c}_{\mathbf{m}}(\mathtt{s},a)\right|_{\mathtt{s}=0}=\frac{\left\langle y_0,\mathcal{H}_{\overline{c}_{\mathbf{m}}(a)}[\check{r}_0,\theta_0]\right\rangle}{\left\langle y_0,(\partial_{c}\mathcal{L}_{c})|_{c=\overline{c}_{\mathbf{m}}(a)}[\check{r}_0]\right\rangle},
\end{equation}
where $\theta_0\in\mathbb{X}_{\mathbf{m},\textnormal{\tiny{even}}}^{s,\sigma}$ is a solution of
$$\mathcal{L}_{\overline{c}_{\mathbf{m}}(a)}[\theta_0]=\mathcal{H}_{\overline{c}_{\mathbf{m}}(a)}[\check{r}_0,\check{r}_0].$$
According to \eqref{Fourier-rep} and \eqref{formula-Hess}, we must take (up to an element of the kernel)
\begin{equation}\label{formula-theta}
	\theta_0(x)=2\mathbf{m}^2M_{2\mathbf{m}}^{-1}\big(a,\overline{c}_{\mathbf{m}}(a)\big)\widetilde{\mathtt{w}}_0\cos(2\mathbf{m}x).
\end{equation}
Now we shall precise, asymptotically in $\mathbf{m}$, if the bifurcation is supercritical of subcritical, namely if the sign of the second term in \eqref{local-coeff} is positive or negative. We take the indices $k\in\{1,2\}$ and $\kappa\in\{-,+\}$ such that
$$\overline{c}_{\mathbf{m}}(a)\underset{\mathbf{m}\to\infty}{\longrightarrow}a_{\kappa}^{[k]}.$$
Using for instance the comatrix formula and the asymptotic \eqref{asymptotic}, we find that
\begin{equation}\label{asymptotic-inv}
	\begin{aligned}
		&M_{2\mathbf{m}}^{-1}\big(a,\overline{c}_{\mathbf{m}}(a)\big)\\
		&\underset{\mathbf{m}\to\infty}{=}\frac{1}{\mathbf{m}^2}\begin{pmatrix}
			\left(a_+^{[1]}-\overline{c}_{\mathbf{m}}(a)\right)^{-1} & 0 & 0 & 0\\
			0 & \left(a_+^{[2]}-\overline{c}_{\mathbf{m}}(a)\right)^{-1} & 0 & 0\\
			0 & 0 & \left(a_-^{[1]}-\overline{c}_{\mathbf{m}}(a)\right)^{-1} & 0\\
			0 & 0 & 0 & \left(a_-^{[2]}-\overline{c}_{\mathbf{m}}(a)\right)^{-1}
		\end{pmatrix}\\
		&\qquad\quad+O\left(\frac{\left(a_{\kappa}^{[k]}-\overline{c}_{\mathbf{m}}(a)\right)^{-1}}{\mathbf{m}^3}\right).
	\end{aligned}
\end{equation}
Combining \eqref{formula-theta}, \eqref{tildettw0} and the asymptotic \eqref{asymptotic-inv}, we infer
\begin{equation}\label{asymptheta0}
	\theta_0(x)\underset{\mathbf{m}\to\infty}{=}2\begin{pmatrix}
		\left(a_+^{[1]}-\overline{c}_{\mathbf{m}}(a)\right)^{-3}\\
		\left(a_+^{[2]}-\overline{c}_{\mathbf{m}}(a)\right)^{-3}\\
		\left(a_-^{[1]}-\overline{c}_{\mathbf{m}}(a)\right)^{-3}\\
		\left(a_-^{[2]}-\overline{c}_{\mathbf{m}}(a)\right)^{-3}
	\end{pmatrix}\cos(2\mathbf{m}x)+O\left(\frac{\left(a_{\kappa}^{[k]}-\overline{c}_{\mathbf{m}}(a)\right)^{-3}}{\mathbf{m}}\right).
\end{equation}
Using one more time \eqref{deriv2} with \eqref{dim1-kerL} and \eqref{asymptheta0}, we find
$$\mathcal{H}_{\overline{c}_{\mathbf{m}}(a)}[\check{r}_0,\theta_0](x)\underset{\mathbf{m}\to\infty}{=}-2\begin{pmatrix}
	\left(a_+^{[1]}-\overline{c}_{\mathbf{m}}(a)\right)^{-4}\\
	\left(a_+^{[2]}-\overline{c}_{\mathbf{m}}(a)\right)^{-4}\\
	-\left(a_-^{[1]}-\overline{c}_{\mathbf{m}}(a)\right)^{-4}\\
	-\left(a_-^{[2]}-\overline{c}_{\mathbf{m}}(a)\right)^{-4}
\end{pmatrix}\sin(2\mathbf{m}x)+O\left(\frac{\left(a_{\kappa}^{[k]}-\overline{c}_{\mathbf{m}}(a)\right)^{-4}}{\mathbf{m}}\right).$$
As a consequence, 
\begin{equation}\label{numerator2}
	\left\langle y_0,\mathcal{H}_{\overline{c}_{\mathbf{m}}(a)}[\check{r}_0,\theta_0]\right\rangle\underset{\mathbf{m}\to\infty}{=}-2\sum_{p\in\{1,2\}\atop\varsigma\in\{-,+\}}^{4}(-1)^{p+1}\big(a_{\varsigma}^{[p]}-\overline{c}_{\mathbf{m}}(a)\big)^{-5}+\left(\frac{\left(a_{\kappa}^{[k]}-\overline{c}_{\mathbf{m}}(a)\right)^{-5}}{\mathbf{m}}\right).
\end{equation}
Finally, gathering \eqref{local-coeff}, \eqref{numerator2} and \eqref{trans:test}, we get
\begin{align*}
	\left.\frac{d^2}{d\mathtt{s}^2}\overline{c}_{\mathbf{m}}(\mathtt{s},a)\right|_{\mathtt{s}=0}\underset{\mathbf{m}\to\infty}{=}-\frac{2}{\mathbf{m}}\left(a_{\kappa}^{[k]}-\overline{c}_{\mathbf{m}}(a)\right)^{-3}+O\left(\frac{\left(a_{\kappa}^{[k]}-\overline{c}_{\mathbf{m}}(a)\right)^{-3}}{\mathbf{m}^2}\right).
\end{align*}
We deduce that
$$\textnormal{the bifurcation is }\begin{cases}
	\textnormal{supercritical,} & \textnormal{if }\overline{c}_{\mathbf{m}}(a)>a_{\kappa}^{[k]},\vspace{0.2cm}\\
	\textnormal{subcritical,} & \textnormal{if }\overline{c}_{\mathbf{m}}(a)<a_{\kappa}^{[k]}.
\end{cases}$$
In particular, by virtue of \eqref{asymptotic-cpm}, \eqref{barc-sym} and \eqref{barc-suc}, in the cases $a\in\mathcal{A}_{\textnormal{sym}}$ and $a\in\mathcal{A}_{\textnormal{suc}}$, one finds the asymptotic local hyperbolic structure as described in Figures \ref{figure hyperbolic sym} and \ref{figure hyperbolic suc}. This concludes the proof of Theorem \ref{thm:IEstates}.
	\section{Large amplitude solutions}\label{sec glo}
	The aim of this section is to apply the analytic global bifurcation theorem (Theorem~\ref{thm BT}) to extend the local solution branches \eqref{local curve} obtained in the previous section. This continuation argument allows us to follow these branches beyond the neighborhood of the bifurcation point and to describe qualitative properties of their global behavior. This provides the proof of the global bifurcation result stated in Theorem~\ref{thm:GB}.\\
			
			We denote
			$$\mathtt{m}(a,c)\triangleq\min\big(\mathtt{m}_1(a),\mathtt{m}_2(a,c)\big),$$
			where
			\begin{align*}
				\mathtt{m}_1(a)&\triangleq\min_{\kappa\in\{-,+\}}\min_{x\in\mathbb{T}}\left|\check{r}_{\kappa}^{[2]}(x)-\check{r}_{\kappa}^{[1]}(x)+\Delta a\right|,\\
				\mathtt{m}_2(a,c)&\triangleq\min_{k\in\{1,2\}\atop\kappa\in\{-,+\}}\min_{x\in\mathbb{T}}\left|\check{r}_{\kappa}^{[k]}(x)+a_{\kappa}^{[k]}-c\right|.
			\end{align*}
			We consider the following open subet of $\mathbb{R}\times \mathbb{X}_{\mathbf{m},\textnormal{\tiny{even}}}^{s,\sigma}$
			$$V(a)\triangleq\Big\{(c,\check{r})\in\mathbb{R}\times\mathbb{X}_{\mathbf{m},\textnormal{\tiny{even}}}^{s,\sigma}\quad\textnormal{s.t.}\quad\mathtt{m}(a,c)>0\Big\}.$$
			In the sequel, we shall denote for any $\rho>0$
			$$\mathbb{B}_{\mathbf{m}}^{s,\sigma}(\rho)\triangleq\Big\{\check{r}\in\mathbb{X}_{\mathbf{m},\textnormal{\tiny{even}}}^{s,\sigma}\quad\textnormal{s.t.}\quad\|\check{r}\|_{s,\sigma}\leqslant\rho\Big\}.$$
			Let us consider the following closed and bounded set defined for any $n\in\mathbb{N}^*$ by
			$$K_{n}(a)\triangleq\Big\{(c,\check{r})\in[-n,n]\times\mathbb{B}_{\mathbf{m}}^{s,\sigma}(n)\quad\textnormal{s.t.}\quad\mathtt{m}(a,c)\geqslant\tfrac{1}{n}\Big\}.$$
			By construction, the following equality holds
			$$V(a)=\bigcup_{n\in\mathbb{N}^*}K_n(a).$$
			We denote
			\begin{equation}\label{scrKn}
				\mathscr{K}_{n}(a)\triangleq\Big\{(c,\check{r})\in K_n(a)\quad\textnormal{s.t.}\quad F(a,c,\check{r})=0\Big\}.
			\end{equation}
			We have the following result.
			\begin{prop}\label{lem hypGB}
				Let $a\in\mathcal{A}$, $s>\tfrac{3}{2}$ and $\sigma>0.$ Then, the following properties hold true.
				\begin{enumerate}[label=(\roman*)]
					\item We have the inclusion $\mathscr{C}_{\textnormal{\tiny{local}}}^{\mathbf{m}}(a)\subset V(a),$ where $\mathscr{C}_{\textnormal{\tiny{local}}}^{\mathbf{m}}(a)$ is the local curve constructed in \eqref{local curve}.
					\item For any $(c,\check{r})\in V(a)$ with $F(a,c,\check{r})=0,$ the operator $d_{\check{r}}F(a,c,\check{r}):\mathbb{X}_{\mathbf{m},\textnormal{\tiny{even}}}^{s,\sigma}\rightarrow\mathbb{Y}_{\mathbf{m},\textnormal{\tiny{odd}}}^{s-1,\sigma}$ is Fredholm with index zero.
					\item For any $n\in\mathbb{N}^*$, the set $\mathscr{K}_n(a)$ is compact in $\mathbb{R}\times \mathbb{X}_{\mathbf{m},\textnormal{\tiny{even}}}^{s,\sigma}.$
				\end{enumerate}
			\end{prop}
			\begin{proof}
				$(i)$ We come back to the notation \eqref{local curve} and denote
				$$\check{r}_{\mathbf{m}}(\mathtt{s},a)\triangleq\left(\check{r}_{\mathbf{m}}^{[1,+]}(\mathtt{s},a),\check{r}_{\mathbf{m}}^{[2,+]}(\mathtt{s},a),\check{r}_{\mathbf{m}}^{[1,-]}(\mathtt{s},a),\check{r}_{\mathbf{m}}^{[2,-]}(\mathtt{s},a)\right).$$
				Since $\Delta a>0$ and $\overline{c}_{\mathbf{m}}(a)\not\in\left\{a_+^{[1]},a_+^{[2]},a_-^{[1]},a_-^{[2]}\right\},$ then, up to taking $\delta$ small enough, we get for any $k\in\{1,2\}$, any $\kappa\in\{-,+\}$ and any $x\in\mathbb{T}$,
				\begin{align*}
					&\left|\check{r}_{\mathbf{m}}^{[2,\kappa]}(\mathtt{s},a)(x)-\check{r}_{\mathbf{m}}^{[1,\kappa]}(\mathtt{s},a)(x)+\Delta a\right|\geqslant\Delta a-C\delta>0,\\
					&\left|\check{r}_{\mathbf{m}}^{[k,\kappa]}(\mathtt{s},a)(x)+a_{\kappa}^{[k]}-\overline{c}_{\mathbf{m}}(\mathtt{s},a)\right|\geqslant\left|a_{\kappa}^{[k]}-\overline{c}_{\mathbf{m}}(a)\right|-C\delta>0.
				\end{align*}
				This proves the inclusion $\mathscr{C}_{\textnormal{\tiny{local}}}^{\mathbf{m}}(a)\subset V(a).$\\
				$(ii)$ Let $(c,\check{r})\in V(a)$ with $F(a,c,\check{r})=0.$ Differentiating \eqref{bif-funct}, we can write
				$$d_{\check{r}}F(a,c,\check{r})=I_{\check{r}}+K_{\check{r}},$$
				where
				$$I_{\check{r}}\triangleq\begin{pmatrix}
					\left(\check{r}_+^{[1]}+a_+^{[1]}-c\right)\partial_{x} & 0 & 0 & 0\\
					0 & \left(\check{r}_+^{[2]}+a_+^{[2]}-c\right)\partial_{x} & 0 & 0\\
					0 & 0 & \left(\check{r}_-^{[1]}+a_-^{[1]}-c\right)\partial_x & 0\\
					0 & 0 & 0 & \left(\check{r}_-^{[2]}+a_-^{[2]}-c\right)\partial_x
				\end{pmatrix}$$
				and
				$$K_{\check{r}}\triangleq M_{\check{r}}+K_{0},\qquad M_{\check{r}}\triangleq\begin{pmatrix}
					\partial_{x}\check{r}_+^{[1]} & 0 & 0 & 0\\
					0 & \partial_x\check{r}_+^{[2]} & 0 & 0\\
					0 & 0 & \partial_x\check{r}_-^{[1]} & 0\\
					0 & 0 & 0 & \partial_x\check{r}_-^{[2]}
				\end{pmatrix},$$
				with $K_{0}$ as in \eqref{lin op 2}. Since $(c,\check{r})\in V(a)$, then in particular
				$$\forall k\in\{1,2\},\quad\forall\kappa\in\{-,+\},\quad\forall x\in\mathbb{T},\quad\check{r}_{\kappa}^{[k]}(x)+a_{\kappa}^{[k]}-c\neq0.$$
				As a consequence, the operator $I_{\check{r}}:\mathbb{X}_{\mathbf{m},\textnormal{\tiny{even}}}^{s,\sigma}\rightarrow \mathbb{Y}_{\mathbf{m},\textnormal{\tiny{odd}}}^{s-1,\sigma}$ is an isomorphism. Now, recall that the compactness of $K_{0}:\mathbb{X}_{\mathbf{m},\textnormal{\tiny{even}}}^{s,\sigma}\rightarrow \mathbb{Y}_{\mathbf{m},\textnormal{\tiny{odd}}}^{s-1,\sigma}$ has already been proved in the previou section. Let us now prove the compactness of $M_{\check{r}}:\mathbb{X}_{\mathbf{m},\textnormal{\tiny{even}}}^{s,\sigma}\rightarrow \mathbb{Y}_{\mathbf{m},\textnormal{\tiny{odd}}}^{s-1,\sigma}.$ Take $(\check{r}_m)_{m\in\mathbb{N}}\in(\mathbb{X}_{\mathbf{m},\textnormal{\tiny{even}}}^{s,\sigma})^{\mathbb{N}}$ bounded in $\mathbb{X}_{\mathbf{m},\textnormal{\tiny{even}}}^{s,\sigma}.$ By compactness (Rellich's Theorem) of the injection $\mathbb{X}_{\mathbf{m},\textnormal{\tiny{even}}}^{s,\sigma}\hookrightarrow\mathbb{X}_{\mathbf{m},\textnormal{\tiny{even}}}^{s-1,\sigma}$, we can find a converging subsequence, namely $(m_p)_{p\in\mathbb{N}}\in\mathbb{N}^{\mathbb{N}}$ and $\check{r}_{\infty}\in\mathbb{X}_{\mathbf{m},\textnormal{\tiny{even}}}^{s-1,\sigma}$ such that
				$$\check{r}_{m_p}\underset{p\to\infty}{\longrightarrow}\check{r}_{\infty}\quad\textnormal{in }\mathbb{X}_{\mathbf{m},\textnormal{\tiny{even}}}^{s-1,\sigma}.$$
				For any $p\in\mathbb{N}$, we introduce
				$$y_p\triangleq M_{\check{r}}[\check{r}_{m_p}]\in\mathbb{Y}_{\mathbf{m},\textnormal{\tiny{odd}}}^{s-1,\sigma}.$$
				Denoting
				\begin{equation}\label{notation rmp}
					\check{r}_{m_p}=\left(\check{r}_p^{[1,+]},\check{r}_p^{[2,+]},\check{r}_p^{[1,-]},\check{r}_p^{[2,-]}\right),
				\end{equation}
				we can write
				$$y_p=\left(\partial_x\check{r}_+^{[1]}\check{r}_p^{[1,+]},\partial_x\check{r}_+^{[2]}\check{r}_p^{[2,+]},\partial_x\check{r}_-^{[1]}\check{r}_p^{[1,-]},\partial_x\check{r}_-^{[2]}\check{r}_p^{[2,-]}\right).$$
				Given $s>\tfrac{3}{2},$ the norm $\|\cdot\|_{s-1,\sigma}$ is a submultiplicative (even if the space is not a Banach algebra because of the parity). Therefore, for any $p,q\in\mathbb{N}$, one has
				\begin{equation}\label{Cauchy}
					\begin{aligned}
						\left\|y_{p}-y_{q}\right\|_{s-1,\sigma}&=\max_{k\in\{1,2\}\atop\kappa\in\{-,+\}}\left\|\partial_{x}r_{\kappa}^{[k]}\left(\check{r}_{m_p}^{[k,\kappa]}-\check{r}_{m_q}^{[k,\kappa]}\right)\right\|_{s-1,\sigma}\\
						&\lesssim\max_{k\in\{1,2\}\atop\kappa\in\{-,+\}}\left\|\check{r}_{\kappa}^{[k]}\right\|_{s,\sigma}\max_{k\in\{1,2\}\atop\kappa\in\{-,+\}}\left\|\check{r}_{m_p}^{[k,\kappa]}-\check{r}_{m_q}^{[k,\kappa]}\right\|_{s-1,\sigma}\\
						&=\|\check{r}\|_{s,\sigma}\|\check{r}_{m_p}-\check{r}_{m_q}\|_{s-1,\sigma}.
					\end{aligned}
				\end{equation}
				The sequence $\left(\check{r}_{m_p}\right)_{p\in\mathbb{N}}$ being convergent in $\mathbb{X}_{\mathbf{m},\textnormal{\tiny{even}}}^{s-1,\sigma},$ it is of Cauchy-type in $\mathbb{X}_{\mathbf{m},\textnormal{\tiny{even}}}^{s-1,\sigma}.$ Combined with the estimate \eqref{Cauchy}, we deduce that the sequence $(y_{p})_{p\in\mathbb{N}}$ is of Cauchy-type (and thus convergent) in the Banach space $\mathbb{Y}_{\mathbf{m},\textnormal{\tiny{odd}}}^{s-1,\sigma}.$ As a consequence, the operator $d_{\check{r}}F(a,c,\check{r}):\mathbb{X}_{\mathbf{m},\textnormal{\tiny{even}}}^{s,\sigma}\rightarrow \mathbb{Y}_{\mathbf{m},\textnormal{\tiny{odd}}}^{s-1,\sigma}$ is a compact perturbation of an isomorphism. Therefore, it is a Fredholm operator with index zero.\\
				$(iii)$ Let $\left(c_m,\check{r}_m\right)_{m\in\mathbb{N}}\in\big(\mathscr{K}_n(a)\big)^{\mathbb{N}}.$ By construction of the set $\mathscr{K}_n(a)$ in \eqref{scrKn}, we have 
				\begin{equation}\label{bnds}
					(c_m)_{m\in\mathbb{N}}\in[-n,n]^{\mathbb{N}},\qquad(\check{r}_m)_{m\in\mathbb{N}}\in\big(\mathbb{B}_{\mathbf{m}}^{s,\sigma}(n)\big)^{\mathbb{N}}
				\end{equation}
				and
				\begin{equation}\label{zeros stepm}
					\forall m\in\mathbb{N},\quad F(a,c_m,\check{r}_m)=0.
				\end{equation}
				From the first property in \eqref{bnds}, we can apply the Bolzano-Weierstrass Theorem and find the existence of $c_{\infty}\in[-n,n]$ and of a subsequence $(m_p)_{p\in\mathbb{N}}$ such that
				$$\lim_{p\to\infty}c_{m_p}=c_{\infty}.$$
				From the second property in \eqref{bnds} and the weak compacity of $\mathbb{B}_{\mathbf{m}}^{s,\sigma}(n)$, there exists $\check{r}_{\infty}\in \mathbb{X}_{\mathbf{m},\textnormal{\tiny{even}}}^{s,\sigma}$ such that, up to an other extraction,
				$$\check{r}_{m_p}\underset{p\to\infty}{\rightharpoonup}\check{r}_{\infty}\quad\textnormal{in }\mathbb{X}_{\mathbf{m},\textnormal{\tiny{even}}}^{s,\sigma}.$$
				Also, using one more time Rellich's theorem (and uniqueness of the weak limit), we find
				$$\forall\, \tfrac{3}{2}<s'<s,\quad\check{r}_{m_p}\underset{p\to\infty}{\longrightarrow}\check{r}_{\infty}\quad\textnormal{in }\mathbb{X}_{\mathbf{m},\textnormal{\tiny{even}}}^{s',\sigma}.$$
				Since $s'>\tfrac{3}{2},$ the pointwise convergence holds. Therefore, we can pass to the limit $p\to\infty$ in the corresponding subsequence of \eqref{zeros stepm} and obtain
				$$F\left(a,c_{\infty},\check{r}_{\infty}\right)=0.$$
				Our next goal is to show that the sequence $(\check{r}_{m_p})_{p\in\mathbb{N}}$ is of Cauchy-type (and thus convergent) in the Banach space $\mathbb{X}_{\mathbf{m},\textnormal{\tiny{even}}}^{s,\sigma}.$ Let $p,q\in\mathbb{N}$. By virtue of \eqref{zeros stepm} we have 
				$$F\left(a,c_{m_p},\check{r}_{m_p}\right)=0=F\left(a,c_{m_q},\check{r}_{m_q}\right).$$
				From \eqref{bif-funct}, using the notation \eqref{notation rmp},
				substracting the equations, we get for any $k\in\{1,2\}$ and any $\kappa\in\{-,+\},$
				$$\partial_{x}\left(\check{r}_{p}^{[k,\kappa]}-\check{r}_{q}^{[k,\kappa]}\right)=\mathcal{I}_{1}^{\,k,\kappa,p,q}+\mathcal{I}_{2}^{\,k,\kappa,p,q},$$
				where
				\begin{align*}
					\mathcal{I}_{1}^{\,k,\kappa,p,q}&\triangleq\frac{\left(c_{m_q}-c_{m_p}+\check{r}_{m_q}^{[k,\kappa]}-\check{r}_{m_p}^{[k,\kappa]}\right)}{\left(r_p^{[k,\kappa]}+a_{\kappa}^{[k]}-c_{m_p}\right)\left(r_q^{[k,\kappa]}+a_{\kappa}^{[k]}-c_{m_q}\right)}\sum_{\alpha\in\{1,2\}\atop\varsigma\in\{-,+\}}\varsigma(-1)^{\alpha}\partial_{x}^{-1}\check{r}_p^{[\alpha,\varsigma]},\vspace{0.2cm}\\
					\mathcal{I}_{2}^{\,k,\kappa,p,q}&\triangleq -\kappa\sum_{\alpha\in\{1,2\}\atop\varsigma\in\{-,+\}}\varsigma(-1)^{\alpha}\frac{\partial_{x}^{-1}\left(\check{r}_p^{[\alpha,\varsigma]}-\check{r}_q^{[\alpha,\varsigma]}\right)}{r_q^{[k,\kappa]}+a_{\kappa}^{[k]}-c_{m_q}}\cdot
				\end{align*}
				By construction of $\mathscr{K}_{n}(a)$, we have
				$$\forall p\in\mathbb{N},\quad\min_{k\in\{1,2\}\atop\kappa\in\{-,+\}}\min_{x\in\mathbb{T}}\left|\check{r}_{p}^{[k,\kappa]}(x)+a_{\kappa}^{[k]}-c_{m_p}\right|\geqslant\tfrac{1}{n}.$$
				Hence, given $s-1<s'<s$
				\begin{align*}
					\left\|\mathcal{I}_{1}^{\,k,\kappa,p,q}\right\|_{s-1,\sigma}&\lesssim_{n,a}\left(\left|c_{m_p}-c_{m_q}\right|+\left\|\check{r}_{p}^{[k,\kappa]}-\check{r}_{q}^{[k,\kappa]}\right\|_{s-1,\sigma}\right)\max_{\alpha\in\{1,2\}\atop\varsigma\in\{-,+\}}\left\|\partial_{x}^{-1}\check{r}_p^{[\alpha,\varsigma]}\right\|_{s-1,\sigma}\\
					&\lesssim_{n,a} \left(\left|c_{m_p}-c_{m_q}\right|+\left\|\check{r}_{p}^{[k,\kappa]}-\check{r}_{q}^{[k,\kappa]}\right\|_{s',\sigma}\right)\max_{\alpha\in\{1,2\}\atop\varsigma\in\{-,+\}}\left\|\check{r}_p^{[\alpha,\varsigma]}\right\|_{s,\sigma}\\
					&\lesssim_{n,a}\left|c_{m_p}-c_{m_q}\right|+\left\|\check{r}_{p}^{[k,\kappa]}-\check{r}_{q}^{[k,\kappa]}\right\|_{s',\sigma}
				\end{align*}
				and
				\begin{align*}
					\left\|\mathcal{I}_{2}^{\,k,\kappa,p,q}\right\|_{s-1,\sigma}&\lesssim_{n,a}\max_{\alpha\in\{1,2\}\atop\varsigma\in\{-,+\}}\left\|\partial_{x}^{-1}\left(\check{r}_p^{[\alpha,\varsigma]}-\check{r}_q^{[\alpha,\varsigma]}\right)\right\|_{s-1,\sigma}\\
					&\lesssim_{n,a}\max_{\alpha\in\{1,2\}\atop\varsigma\in\{-,+\}}\left\|\check{r}_p^{[\alpha,\varsigma]}-\check{r}_q^{[\alpha,\varsigma]}\right\|_{s',\sigma}.
				\end{align*}
				The sequences $\big(c_{m_p}\big)_{p\in\mathbb{N}}$ and $(\check{r}_{m_p})_{p\in\mathbb{N}}$ being convergent in $\mathbb{R}$ and $\mathbb{X}_{\mathbf{m},\textnormal{\tiny{even}}}^{s',\sigma}$ respectively, in particular  they are of Cauchy-type in the corresponding spaces. This gives the desired result.
				Thus, for any $n\in\mathbb{N}^*$, the set $\mathscr{K}_n(a)$ is compact in $\mathbb{R}\times\mathbb{X}_{\mathbf{m},\textnormal{\tiny{even}}}^{s,\sigma}.$ This ends the proof of Proposition \ref{lem hypGB}.
			\end{proof}
			With this in hand, we can conclude.
			\begin{proof}[Proof of Theorem \ref{thm:GB}] The Proposition \ref{lem hypGB} allows to apply the Theorem \ref{thm BT} which provides the existence of a global continuation curve $\mathscr{C}_{\textnormal{\tiny{global}}}^{\mathbf{m}}(a)$ satisfying
				$$\mathscr{C}_{\textnormal{\tiny{local}}}^{\mathbf{m}}(a)\subset\mathscr{C}_{\textnormal{\tiny{global}}}^{\mathbf{m}}(a)\triangleq\Big\{\big(\overline{c}_{\mathbf{m}}(\mathtt{s},a),\check{r}_{\mathbf{m}}(\mathtt{s},a)\big),\quad\mathtt{s}\in\mathbb{R}\Big\}\subset V(a)\cap F(a,\cdot,\cdot)^{-1}\big(\{0\}\big).$$
				Moreover, the curve $\mathscr{C}_{\textnormal{\tiny{global}}}^{\mathbf{m}}(a)$ admits locally around each of its points a real-analytic reparametrization. In addition, one of the following alternatives occurs
				\begin{itemize}
					\item [$(A1)$] (Loop) There exists $T_{\mathbf{m}}(a)>0$ such that 
					$$\forall\mathtt{s}\in\mathbb{R},\quad \overline{c}_{\mathbf{m}}\big(\mathtt{s}+T_{\mathbf{m}}(a),a\big)=\overline{c}_{\mathbf{m}}(\mathtt{s},a)\qquad\textnormal{and}\qquad \check{r}_{\mathbf{m}}\big(\mathtt{s}+T_{\mathbf{m}}(a),a\big)=\check{r}_{\mathbf{m}}(\mathtt{s},a).$$
					\item [$(A2)$] One the following limits holds (possibly simultaneously)
					\begin{enumerate}[label=\textbullet]
						\item (Blow-up) $\displaystyle\lim_{\mathtt{s}\to\pm\infty}\frac{1}{1+\left|\overline{c}_{\mathbf{m}}(\mathtt{s},a)\right|+\|\check{r}_{\mathbf{m}}(\mathtt{s},a)\|_{s,\sigma}}=0.$
						\item (Collision of the boundaries) $\displaystyle\lim_{\mathtt{s}\to\pm\infty}\min_{\kappa\in\{-,+\}}\min_{x\in\mathbb{T}}\left|\check{r}_{\mathbf{m}}^{[2,\kappa]}(\mathtt{s},a)(x)-\check{r}_{\mathbf{m}}^{[1,\kappa]}(\mathtt{s},a)(x)+\Delta a\right|=0.$
						\item (Degeneracy) $\displaystyle\lim_{\mathtt{s}\to\pm\infty}\min_{k\in\{1,2\}\atop\kappa\in\{-,+\}}\min_{x\in\mathbb{T}}\left|\check{r}_{\mathbf{m}}^{[k,\kappa]}(\mathtt{s},a)(x)+a_{\kappa}^{[k]}-\overline{c}_{\mathbf{m}}(\mathtt{s},a)\right|=0.$
					\end{enumerate}
				\end{itemize}
				This achieves the proof of Theorem \ref{thm:GB}.
			\end{proof}
	\appendix
	\section{Toolkit in bifurcation}
	This appendix summarizes the bifurcation theory underlying our analysis. We begin with the classical Crandall–Rabinowitz local bifurcation theorem \cite{CR71} (see also \cite[p.~15]{K11}), stated here in an analytic framework that is better suited to our purposes. In addition, we incorporate supplementary hypotheses that guarantee the occurrence of a pitchfork-type bifurcation and the effective orientation of the branch. For further background and detailed discussions, we refer the reader to the works of Shi \cite{S99} and Liu–Shi \cite{LS22}.
	\begin{theo}\label{thm CR+S}
		\textbf{(Analytic local bifurcation + pitchfork property)}\\
		Let $X$ and $Y$ be two Banach spaces. Let $(p_0,x_0)\in\mathbb{R}\times X$ and $V$ be a neighborhood of $(p_0,x_0)$ in $\mathbb{R}\times X.$ Consider a real-analytic function $F:V\rightarrow Y$ and denote for $(p,x_0)\in V$,
		$$\mathcal{L}_p\triangleq d_xF(p,x_0),\qquad\mathcal{H}_p\triangleq d_x^2F(p,x_0),\qquad\mathcal{C}_p\triangleq d_x^3F(p,x_0).$$
		Assume the following hypothesis.
		\begin{itemize}
			\item [$(L1)$] $\forall(p,x_0)\in V,\quad F(p,x_0)=0.$
			\item [$(L2)$] The operator $\mathcal{L}_{p_0}$ is a Fredholm with zero index and one dimensional kernel
			$$\dim\big(\ker(\mathcal{L}_{p_0})\big)=1=\textnormal{codim}\big(R(\mathcal{L}_{p_0})\big),\qquad\ker(\mathcal{L}_{p_0})=\mathtt{span}(\mathtt{v}_0).$$
			\item [$(L3)$] Transversality: 
			$$(\partial_{p}\mathcal{L}_p)|_{p=p_0}[\mathtt{v}_0]\not\in R(\mathcal{L}_{p_0}).$$
		\end{itemize}
		If we decompose
		$$X=\mathtt{span}(\mathtt{v}_0)\oplus Z,$$
		then there exist two real-analytic functions
		$$p:(-\delta,\delta)\rightarrow\mathbb{R}\qquad\textnormal{and}\qquad z:(-\delta,\delta)\rightarrow Z,\qquad\textnormal{with}\qquad\delta>0,$$
		such that
		$$p(0)=p_0,\qquad z(0)=0$$
		and the set of zeros of $F$ in $V$ is the union of two curves 
		$$V\cap F^{-1}(\{0\})=\big\{(p,x_0)\in V\big\}\cup\mathscr{C}_{\textnormal{\tiny{local}}},\qquad\mathscr{C}_{\textnormal{\tiny{local}}}\triangleq\big\{\big(p(\mathtt{s}),x_0+\mathtt{s}\mathtt{v}_0+\mathtt{s}z(\mathtt{s})\big),\quad|\mathtt{s}|<\delta\big\}.$$
		Assume in addition that
		$$\mathcal{H}_{p_0}[\mathtt{v}_0,\mathtt{v}_0]\in R\big(\mathcal{L}_{p_0}\big).$$
		Then $p'(0)=0$ and if we denote
		$$R\big(\mathcal{L}_{p_0}\big)=\ker(l)\qquad\textnormal{ for some }l\in Y^*,$$
		then we have
		$$p''(0)=\frac{3\big\langle l\,,\,\mathcal{H}_{p_0}[\mathtt{v}_0,\theta_0]\big\rangle-\big\langle l\,,\,\mathcal{C}_{p_0}[\mathtt{v}_0,\mathtt{v}_0,\mathtt{v}_0]\big\rangle}{3\big\langle l\,,\,(\partial_p \mathcal{L}_{p})|_{p=p_0}[\mathtt{v}_0]\big\rangle},$$
		where $\theta_0$ is a solution of
		$$\mathcal{L}_{p_0}[\theta_0]=\mathcal{H}_{p_0}[\mathtt{v}_0,\mathtt{v}_0].$$
		If $p''(0)\neq 0,$ we say that the bifurcation is of pitchfork-type. More precisely, the condition $p''(0)>0$ (resp. $p''(0)<0$) is called supercritical (resp. subcritical) bifurcation.
	\end{theo}
	We conclude by recalling a global bifurcation result that complements the preceding local analysis and allows the continuation of solution branches beyond the small-amplitude regime. More precisely, we state a classical global bifurcation theorem, originally due to Dancer \cite{D73} and later developed in the monograph of Buffoni and Toland \cite[Thm.~9.1.1]{BT03}. This result provides a framework for extending locally constructed branches and characterizes the possible global behaviors of the solution set, such as unbounded growth or the occurrence of loops. For convenience, we adopt the formulation given in \cite[Thm.~4]{CSV16}, which is particularly well suited to our setting.
	\begin{theo}\label{thm BT}
		\textbf{(Analytic global bifurcation)}
		Let $X,Y,V,F$ as in Theorem \ref{thm CR+S} such that the assumptions $(L1)$, $(L2)$ and $(L3)$ are satisfied. Assume also the following additional properties.
		\begin{itemize}
			\item [$(G1)$] For any $(p,x)\in V$ such that $F(p,x)=0$, the operator $d_xF(p,x)$ is a Fredholm operator of index $0.$
			\item [$(G2)$] There exists $(K_n)_{n\in\mathbb{N}}$ such that
			\begin{enumerate}[label=\textbullet]
				\item Exhaustive property
				$$V=\bigcup_{n\in\mathbb{N}}K_n.$$
				\item For any $n\in\mathbb{N},$ the set $K_n$ is bounded and closed in $\mathbb{R}\times X.$
				\item For any $n\in\mathbb{N},$ the set $K_n\cap F^{-1}\big(\{0\}\big)$ is compact in $\mathbb{R}\times X.$  
			\end{enumerate}
		\end{itemize}
		Then there exists a unique (up to reparametrization) continuous curve $\mathscr{C}_{\textnormal{\tiny{global}}}$ such that
		$$\mathscr{C}_{\textnormal{\tiny{local}}}\subset\mathscr{C}_{\textnormal{\tiny{global}}}\triangleq \Big\{\big(p(\mathtt{s}),x(\mathtt{s})\big),\quad\mathtt{s}\in\mathbb{R}\Big\}\subset V\cap F^{-1}\big(\{0\}\big).$$
		Moreover, $\mathscr{C}_{\textnormal{\tiny{global}}}$ admits locally around each of its points a real-analytic parametrization. In addition, one of the following alternatives occurs
		\begin{itemize}
			\item [$(A1)$] there exists $T>0$ such that
			$$\forall\mathtt{s}\in\mathbb{R},\quad p(\mathtt{s}+T)=p(\mathtt{s})\qquad\textnormal{and}\qquad x(\mathtt{s}+T)=x(\mathtt{s}).$$
			\item [$(A2)$] for any $n\in\mathbb{N},$ there exists $\mathtt{s}_n>0$ such that 
			$$\forall\mathtt{s}>\mathtt{s}_n,\quad\big(p(\mathtt{s}),x(\mathtt{s})\big)\not\in K_n.$$
		\end{itemize}
	\end{theo}

\end{document}